\documentclass{amsart}  
\usepackage{latexsym,amsthm,xspace,enumerate,amsfonts,amsmath,amssymb,amscd}
\usepackage[arrow, matrix, curve]{xy}
\usepackage{}

\newtheorem{proposition}{\textbf{Proposition}}
\newtheorem{lemma}[proposition]{\textbf{Lemma}}
\newtheorem{corollary}[proposition]{\textbf{Corollary}}
\newtheorem{theorem}[proposition]{\textbf{Theorem}}

\theoremstyle{definition}
\newtheorem{definition}[proposition]{\textbf{Definition}}

\newtheorem{example}[proposition]{\textbf{Example}}
\newtheorem{remark}[proposition]{\textbf{Remark}}
\newtheorem{acknowledgements}{\textbf{Acknowledgements}}

\newcommand{\note}[1]{\marginpar{\raggedright\if@twoside\ifodd\c@page\raggedleft\fi\fi\sf\scriptsize \red{RMK: #1}}}

\newcommand{\lie}[1]{\operatorname{\mathfrak{#1}}}

\newcommand{\su}{\lie{su}}

\newcommand{\un}{\lie u}

\newcommand{\hf}{\lie h}

\newcommand{\kf}{\lie k}

\newcommand{\gf}{\lie g}

\newcommand{\mf}{\lie m}

\newcommand{\p}{\lie p}

\newcommand\C{{\mathbb C}}

\newcommand{\PP}{{\mathbb P}}

\newcommand{\R}{{\mathbb R}}

\newcommand{\B}{{\mathbb B}}

\newcommand{\sM}{{\mathcal{M}}}   % moduli space

\newcommand{\cx}{{\mathbb C}}

\newcommand{\ad}{\operatorname{ad}}
\newcommand{\Ad}{\operatorname{Ad}}
\newcommand{\tr}{\operatorname{tr}}

\newcommand{\Jac}{\operatorname{Jac}}

\newcommand{\Mat}{\operatorname{Mat}}

\newcommand{\ol}{\overline}

\newcommand{\oZ}{{\mathbb{Z}}}

\newcommand{\oP}{{\mathbb{P}}}

\newcommand{\oR}{{\mathbb{R}}}

\numberwithin{equation}{section}
\numberwithin{proposition}{section}

\begin{document}
\title[Nahm--Schmid equations]{The Nahm--Schmid equations \\and Hypersymplectic Geometry}

\author{Roger Bielawski, Nuno M. Rom\~ao, Markus R\"oser}
\address{Institut f\"ur Differentialgeometrie,
Leibniz Universit\"at Hannover,
Welfengarten 1, 30167 Hannover, Germany}
\email{bielawski@math.uni-hannover.de \\ roeser@math.uni-hannover.de}
\address{Institut f\"ur Mathematik, Universit\"at Augsburg, 86135 Augsburg, Germany}
\email{Nuno.Romao@math.uni-augsburg.de}

\begin{abstract} 
We explore the geometry of the Nahm--Schmid equations, a version of Nahm's equations in split signature. Our discussion ties up different aspects of their integrable nature: dimensional reduction
from the Yang--Mills anti-self-duality equations, explicit solutions, Lax-pair formulation, conservation laws and spectral curves, as well as their relation to hypersymplectic geometry.
\end{abstract}

\maketitle

\tableofcontents
\section{Introduction}
Hypersymplectic geometry~\cite{Hitchin:1990,DanSwa} can informally be thought of as a pseudo-Riemannian analogue of  hyperk\"ahler geometry, in which the role of the quaternions $\mathbb H$ is played instead by the algebra $\B$  of the \emph{split quaternions}. The generators of $\B$ (over $\mathbb{R}$) are denoted ${\rm i,s,t}$ and satisfy the relations
$${\rm i}^2 = -1, \quad {\rm s}^2 = 1 = {\rm t}^2, \quad {\rm is} = {\rm t} = -{\rm si}.$$
As an $\mathbb R$-algebra, $\B$ is isomorphic to the associative algebra ${\rm Mat}_2(\R)$ of real $2\times 2$ matrices, which is the split real form of the complex algebra ${\rm Mat}_2(\C)$. 
A hypersymplectic manifold is then  a manifold  with an action of $\B$ on its tangent bundle which is compatible with a given pseudo-Riemannian metric of split signature;
its dimension is necessarily a multiple of four.

\begin{definition}
A \emph{hypersymplectic manifold} is a quintuple $(M^{4k},g,I,S,T)$ where $g$ is a pseudo-Riemannian metric on $M$ of signature $(2k,2k)$ and $I,S,T\in \Gamma(\mathrm{End}({\rm T}M))$ are parallel  skew-adjoint endomorphisms satisfying the relations 
$$I^2 = - 1, \quad S^2 = 1 = T^2,\quad IS = T =-SI.$$
\end{definition}

The parallel endomorphism $S$ (and similarly $T$) splits ${\rm T}M$ into its $(\pm1)$-eigen\-bundles, which are of equal rank and integrable in the sense of Frobenius. A hypersymplectic manifold $M$ therefore locally splits (essentially, in a canonical way) as a product $M_+\times M_-$, where the two factors have the same dimension $2k$. For this reason, the endomorphism $S$ is also called \emph{a local product structure} in the literature.

Given a hypersymplectic manifold, we obtain a triple of symplectic forms
\begin{equation} \label{tripleomega}
\omega_I = g(I\cdot, \cdot),\; \omega_S = g(S\cdot, \cdot),\; \omega_T=g(T\cdot, \cdot)
\end{equation}
in analogy to the hyperk\"ahler situation.
The holonomy group of a hypersymplectic manifold is the real symplectic group ${\rm Sp}(2k,\oR)$, which in particular implies that hypersymplectic manifolds are Ricci-flat. Hypersymplectic and hyperk\"ahler structures  have the same complexification, so it is not surprising that many results and constructions from hyperk\"ahler geometry have hypersymplectic analogues.
For example, one can show just like in the hyperk\"ahler case that the endomorphisms $I,S,T$ are parallel and integrable if and only if the associated 2-forms $\omega_I,\omega_S,\omega_T$ are all closed.  
In fact, the hypersymplectic structure, i.e. the quadruple $(g,I,S,T)$, is completely determined by the corresponding triple of symplectic forms (\ref{tripleomega}). For instance, we have $S=\omega_T^{-1}\circ\omega_I$, if we interpret the $\omega_i$'s as isomorphisms $\omega_i: {\rm T}M\to {\rm T}^\ast M$.

There is also a hypersymplectic quotient construction analogous to (but more pathological than) the hyperk\"ahler quotient of \cite{HKLR:1987}. Quotients of this sort constitute a basic source of hypersymplectic structures, and they may occur in dimensional reduction from split signature in various contexts in mathematical physics --- see for example~\cite{Hull} for an application to the study of sigma-models with extended supersymmetry. 
The basic result is the following: 

\begin{proposition}[\cite{Hitchin:1990}] Let $(M,g,I,S,T)$ be a hypersymplectic manifold with an action of a Lie group $G$ which is Hamiltonian with respect to each symplectic form $\omega_i, i\in \{I,S,T\}$, with hypersymplectic moment map 
$$\mu=(\mu_I,\mu_S,\mu_T): M\to \gf^\ast \otimes\R^3.$$
Assume that the $G$-action is free and proper on $\mu^{-1}(0)$, that $0$ is a regular value of $\mu$, and that the metric restricted to the $G$-orbits in $\mu^{-1}(0)$ is non-degenerate. Then the quotient 
$$\mu^{-1}(0)/G$$
carries in a natural way a hypersymplectic structure. 
\end{proposition}

In many situations, one can obtain a smooth quotient manifold $\mu^{-1}(0) /G$. However,  this will typically carry a hypersymplectic structure only on the complement of a \emph{degeneracy locus}. On the other hand, if $G$ acts freely and properly on $\mu^{-1}(0)$, then the non-degeneracy assumption implies the smoothness of the quotient $\mu^{-1}(0)/G$ just like in the proof of the hyperk\"ahler version of the construction. 

It is our aim in this paper to apply the hypersymplectic quotient construction in an infinite-dimensional setting to obtain a hypersymplectic structure on a suitable open subset of the product manifold $G\times \gf^3$. Our approach is closely analogous to Kronheimer's construction of a hyperk\"ahler metric on ${\rm T}^\ast G^\C$ in \cite{Kronheimer:1988}. Building in part on work by Matsoukas \cite{Matsoukas:2010}, we interpret $G\times \gf^3$ as a moduli space of solutions to the {\em Nahm--Schmid equations} (see Section~\ref{NahmSchmidEq} below), a ``hypersymplectic version'' of Nahm's equations~\cite{Nahm} which can be viewed as a dimensional reduction of the anti-self-dual Yang-Mills equations in split signature. To our best knowledge, these equations first arose in the work of Schmid on deformations of complex manifolds \cite{Schm}, which predates the advent of Nahm's equations in gauge theory~\cite{Nahm,AtiyahHitchin:1988}.

The article is organised as follows. 
In Section 2 we introduce the Nahm--Schmid equations as a dimensional reduction of the anti-self-dual Yang-Mills equations in split signature, and derive some immediate properties of their solutions with values in a compact Lie algebra $\gf$. The heart of the article is Section 3, where we examine the hypersymplectic geometry of the (framed) moduli space of solutions over the unit interval. We describe the corresponding
degeneracy locus and investigate the complex and product structures of our moduli space. The results that we obtain are akin to those in \cite{Roeser:2014}, where the moduli space of solutions to the gauge-theoretic harmonic map equations was studied; note that these equations also arise by dimensional reduction of the anti-self-dual Yang Mills equations in split signature.  Then in Section 4 we consider the Nahm--Schmid equations from the point of view of integrable systems  \cite{AdlvMoe1,AdlvMoe2}. We discuss the relevant twistor space (as a particular case of a more general construction by Bailey and Eastwood~\cite{BE}), as well as a Lax pair formulation of 
 the equations, in particular the associated {\em spectral curve}. Finally, in Section 5 we discuss other moduli spaces of solutions to the Nahm--Schmid equations. We show that the moduli spaces of solutions on a half-line (with the limit satisfying a certain stability condition) is a hypersymplectic analogue of the hyperk\"ahler metrics on adjoint orbits in $\gf^\cx$. We also briefly discuss hypersymplectic quotients of the moduli space of solutions over the unit interval by subgroups of $G\times G$.
 
\begin{acknowledgements} 
Parts of this work are contained in the third author's DPhil thesis \cite{RoeserDPhil} and he wishes to thank his supervisor, Prof.\ Andrew Dancer, for his guidance and support. The second author would like to thank Paul Norbury for early discussions related to material contained in this paper, and FCT for
financial support under the grant BPD/14544/03.
\end{acknowledgements} 
 
\section{The Nahm--Schmid equations}\label{NahmSchmidEq}

We use $\R^{2,2}$ to denote $\R^4$ equipped with the pseudo-Riemannian metric $\mathrm dx_0^2+\mathrm dx_1^2-\mathrm dx_2^2-\mathrm dx_3^2$. Let $G$ be a compact Lie group with Lie algebra $\gf$, on which we choose a $\mathrm{Ad}$-invariant inner product $\langle \cdot,\cdot\rangle$. 

Solutions to the Nahm--Schmid equations will correspond to $G$-connections  
$$\nabla = \mathrm d + \sum_{i=0}^3 T_i\, \mathrm dx_i$$
on $\mathbb R^{2,2}$ that are anti-self-dual (with respect to the standard orientation given by $\mathrm dx_0\wedge\mathrm dx_1\wedge\mathrm dx_2\wedge\mathrm dx_3$)
and also invariant under translation in the variables $x_1, x_2, x_3$. That is, the connection matrices $T_i$ depend only on the (real) coordinate $x_0$. We recall that the anti-self-duality equations on $\mathbb R^{2,2}$  are 
\begin{eqnarray*}
\ [ \nabla_0, \nabla_1] &=& -[\nabla_2, \nabla_3], \ \\
\ [ \nabla_0,\nabla_2] &=& -[\nabla_1, \nabla_3],\ \\
\ [\nabla_0, \nabla_3] &=& [\nabla_1, \nabla_2].\   
\end{eqnarray*}
Since we assume that the $T_i$ only depend on $x_0$ (which we shall also denote by $t$), we have for the covariant partial derivatives in these equations
$$\nabla_0 = \frac{\mathrm d}{\mathrm dt} + T_0 \qquad \text{and} \qquad \nabla_i = T_i,\  \ i = 1,2,3.$$
Thus, we arrive at the following definition.
\begin{definition}\label{Def:NahmSchmidEq}
Let $\gf$ be the Lie algebra of a compact Lie group.  A quadruple of $\mathfrak g$-valued functions $(T_0,T_1,T_2,T_3)$, $T_i: \R \to \gf$  satisfies the  \emph{Nahm--Schmid equations} if 
\begin{eqnarray} \label{NahmSchmidUNred}
\dot T_1 + [T_0,T_1] &=& -[T_2,T_3],\nonumber \\
\dot T_2 + [T_0,T_2] &=& [T_3,T_1], \\ 
\dot T_3 + [T_0,T_3] &=& [T_1,T_2]. \nonumber
\end{eqnarray}
\end{definition}

The solutions to the Nahm--Schmid equations are invariant under gauge transformations, i.e. functions $u:\R \rightarrow G$ acting by 
\begin{equation}\label{gaugeaction} u.(T_0,T_1,T_2,T_3) := (uT_0u^{-1}-\dot uu^{-1}, uT_1u^{-1}, uT_2u^{-1}, uT_3u^{-1}).\end{equation}
Note that we can always find a gauge transformation that solves the ODE  $$uT_0u^{-1}-\dot uu^{-1} = 0.$$
This allows us to transform any solution $(T_0,T_1,T_2,T_3)$ of (\ref{NahmSchmidUNred}) into a solution with $T_0=0$. For such a solution $\mathcal T = (0, T_1, T_2, T_3)$, the triple $(T_1,T_2,T_3)$ satisfies the \emph{reduced Nahm--Schmid equations}
\begin{eqnarray} \label{NahmSchmidred}
\dot T_1 &=& -[T_2,T_3],  \nonumber \\
\dot T_2  &=& [T_3,T_1],  \\ 
\dot T_3  &=& [T_1,T_2] . \nonumber
\end{eqnarray}

An important property of this system of equations is the following.

\begin{proposition}\label{Conserved} Let $(T_0,T_1,T_2,T_3)$ be a solution to the Nahm--Schmid equations. Then the following gauge-invariant quantities are conserved:
$$\|T_1\|^2 + \|T_i\|^2,  \quad i=2,3 \qquad \text{ and }\qquad \langle T_i,T_j\rangle, \quad i\neq j, \;\;  i,j\geq 1.$$
\end{proposition}
\begin{proof}
This follows from a direct computation, using the equations and the invariance of the inner product.
\end{proof}
We shall give a natural interpretation of this statement in Section \ref{Lax}.
An immediate consequence, which is in contrast to the behaviour of solutions of the usual Nahm equations~\cite{Nahm}, is a global existence result for solutions that can be brought to the reduced form $\mathcal T$ above (see also \cite{Matsoukas:2010}):

\begin{corollary}\label{GlobalReg}
Any solution of the reduced Nahm--Schmid equations \ref{NahmSchmidred} exists for all time.
\end{corollary}
\begin{proof}
Proposition~\ref{Conserved} shows that we have a conserved quantity 
$$C = 2\|T_1\|^2 + \|T_2\|^2 + \|T_3\|^2;$$
note that for any $t\in\R$ one has
$$\|T_i(t)\|^2 \leq C.$$
Thus the solution is uniformly bounded. The assertion now follows from the fact that, if a solution to an ODE only exists for finite time, it has to leave every compact set. 
\end{proof}

\begin{remark}
Put on $\gf\times\gf\times\gf$ the indefinite metric coming from the identification with $\gf\otimes \R^{1,2}$ and consider the function 
$$\phi: \gf\times\gf\times\gf\to \R, \qquad \phi(\xi_1,\xi_2,\xi_3) := \langle [\xi_1,\xi_2],\xi_3\rangle.$$
The right-hand side of the reduced Nahm--Schmid equations (\ref{NahmSchmidred}) are then the negative of the gradient of $\phi$, with respect to this metric. We know from Proposition \ref{GlobalReg} that solutions to this gradient flow are bounded, and so exist for all times. It follows moreover that the non-trivial trajectories of this flow are contained in the compact submanifolds
$$M_C = \{(\xi_1,\xi_2,\xi_3) \;|\;  \, 2\|\xi_1\|^2 + \|\xi_2\|^2 + \|\xi_3\|^2 = C\}\subset \gf^3,$$
for suitable $C>0$.
\end{remark}

\section{The moduli space over $[0,1]$}

In this section, we study solutions of the Nahm--Schmid equations over a compact interval $U\subset \R$, which we fix to be $U=[0,1]$ for convenience. Again, $G$ denotes any compact Lie group and $\mathfrak g$ stands for its Lie algebra.

\subsection{The moduli space as a manifold\label{2}}
We denote by $\mathcal A$ the space of all quadruples of $C^1$-functions on $U$ with values in $\gf$, i.e. the (affine) Banach manifold
$$\mathcal A = \{(T_0,T_1,T_2,T_3): U\to \gf\otimes \R^4\ | \ \text{$T_i$ differentiable of class $C^1$}\} \cong C^1(U,\gf)\otimes \R^4$$
equipped with the norm
$$\|T\|_{C^1} := \sum_{i=0}^3 \|T_i\|_{C^1},$$
where we use the usual $C^1$-norm $\|T_i\|_{C^1} := \|T_i\|_{\sup}+ \|\dot T_i\|_{\sup}$ on each component.

\begin{proposition}\label{NSsmooth}
The set  of  solutions to the Nahm--Schmid equations is a smooth Banach submanifold of $\mathcal A$. 
\end{proposition}
\begin{proof}
This follows just like the analogous statement in the hyperk\"ahler case \cite{Kronheimer:1988}. Let $\mu:\mathcal A\to C^0(U,\gf)\otimes \R^3$ denote the difference of the RHS and the LHS in \eqref{NahmSchmidUNred}, so that the set of solutions to the Nahm--Schmid equations is given by $\mu^{-1}(0)$. Owing to the implicit function theorem, we have to check that, for any solution $\mathcal T= (T_0,T_1,T_2,T_3)\in\mu^{-1}(0)$, the linearisation 
$$\mathrm d\mu_\mathcal T: {\rm T}_\mathcal T\mathcal A \to C^0(U,\gf)\otimes \R^3$$
has a bounded right-inverse. That is, to $\zeta = (\zeta_1,\zeta_2,\zeta_3)\in C^0(U,\gf)\otimes \R^3$ we must associate a solution
 $X = (X_0,X_1,X_2,X_3)\in {\rm T}_\mathcal T\mathcal A$ of the system of linear ODEs
\begin{eqnarray*}
\dot X_1 + [T_0,X_1] + [X_0, T_1] + [T_2,X_3] + [X_2,T_3] &=& \zeta_1 \\
\dot X_2 + [T_0,X_2] + [X_0, T_2] - [T_3,X_1] - [X_3,T_1] &=& \zeta_2 \\
\dot X_3 + [T_0,X_3] + [X_0, T_3] - [T_1,X_2] - [X_1,T_2] &=& \zeta_3.
\end{eqnarray*}
We note, just like in \cite{Kronheimer:1988}, that for any $\zeta$ this system has a unique solution with $X_0\equiv 0$, $X_i(0) =0$, $i=1,2,3$, and deduce the existence of the required right-inverse.
\end{proof}

We now want to quotient by a suitable gauge group. We consider the Banach Lie group 
$$\mathcal G := C^2(U,G)$$
of gauge transformations acting on $\mathcal A$.
This has a normal subgroup consisting of gauge transformations that are equal to the identity at the endpoints of $U$:
$$\mathcal G_{00} = \{u\in C^2(U,G)\ | \ u(0) = 1_G = u(1)\}.$$
The Lie algebras of $\mathcal G$ and $\mathcal G_{00}$ are given by 
$$\mathrm{Lie}(\mathcal G) = C^2(U,\gf)$$
and
$$\mathrm{Lie}(\mathcal G_{00}) = \{\xi\in C^2(U,\gf)\ | \ \xi(0) = 0 = \xi(1)\}.$$

The action of $\mathcal G$ (and also of $\mathcal G_{00}$) on $\mathcal A$ is smooth and given by \eqref{gaugeaction}. Furthermore:
\begin{lemma}\label{freeproper}
The gauge group $\mathcal G_{00}$ acts freely and properly on $\mathcal A$. 
\end{lemma}
\begin{proof}
 Suppose we have $u\in \mathcal G_{00}$ and $\mathcal T\in\mathcal A$ such that $u.\mathcal T=\mathcal T$. Looking at the $T_0$-component, we see that $u$ solves the intial value problem
$$\dot uu^{-1} = uT_0u^{-1} - T_0, \qquad u(0) = 1_G.$$
Its unique solution is the constant map $u(t)\equiv 1_G$ and hence the action is free. To see that it is proper suppose that we have a sequence $(\mathcal T^m)\subset \mathcal A$ converging to $\mathcal T\in \mathcal A$ and a sequence $(u_m)\subset \mathcal G_{00}$ such that $u_m\mathcal T^m\to \tilde{\mathcal T}\in \mathcal A$. Then 
\begin{equation}\label{convr} u_mT^m_0u_m^{-1}-\dot u_mu_m^{-1}\to \tilde T_0.
\end{equation}
Since $T_0^m$ also converges to $T_0$ in the $C^1$-topology, it follows firstly from the Arzel\`a-Ascoli theorem that a subsequence of $u_m$ converges in the $C^0$-topology, and then a repeated use of \eqref{convr} shows that the convergence is actually in the $C^2$-topology.  Thus the action of $\mathcal G_{00}$ is proper.
\end{proof}

We are now ready to introduce our main object of study.
\begin{definition} \label{moduli}
The  \emph{moduli space of solutions to the Nahm--Schmid equations} is 
\begin{equation}\label{modulispace}
\mathcal M := \{\mathcal T\in\mathcal A\ | \ \text{$\mathcal T$ solves equations (\ref{NahmSchmidUNred})} \}/\mathcal G_{00}.
\end{equation}
\end{definition}

\begin{theorem}\label{MGg3}
$\mathcal M$ is a smooth Banach manifold diffeomorphic to $G\times \gf\times\gf\times\gf$.  
\end{theorem}
\begin{proof}
Owing to Lemma \ref{freeproper}, $\mathcal M$, equipped with the quotient topology, is Hausdorff. Let $\mathcal T= (T_0,T_1,T_2,T_3)$ be a solution of the Nahm--Schmid equations. We need to construct a slice to the action of $\mathcal G_{00}$ at $\mathcal T$. Let $u_0$ be the unique gauge transformation such that $T_0= -\dot u_0 u_0^{-1}$
and $u_0(0)=1_G$.  Write $\mathcal A_{\rm NS}$ for the set of solutions to the  Nahm--Schmid equations and define
$$ S_{\mathcal T}=\bigl\{ \mathcal S=(S_0,S_1,S_2,S_3)\in\mathcal A_{\rm NS}\ | \  S_0=u_0X_0u_0^{-1}-\dot u_0 u_0^{-1}, \enskip X_0\in\gf, \enskip \|X_0\|< \epsilon\bigr\},$$
where $\epsilon>0$ will be determined later.
If $u.\mathcal S\in S_{\mathcal T}$ for an  $\mathcal S\in S_{\mathcal T}$ and a $u\in \mathcal G_{00}$, then $g=u_0^{-1}uu_0\in \mathcal G_{00}$ and $gX_0g^{-1}-\dot gg^{-1}=Y_0\in \gf$. The unique solution with $g(0)=1$ is $\exp(-tY_0)\exp(tX_0)$ and hence $g(t)\equiv 1_G$ if $\epsilon$ is small enough. It follows then that  
$u(t)\equiv 1_G$ and the map
\begin{equation}\label{slice} S_{\mathcal T}\times \mathcal G_{00}\to \mathcal A_{\rm NS},\quad (\mathcal S,u)\mapsto u.\mathcal S\end{equation}
is injective for $\epsilon$  small enough. The inverse map is given explicitly as follows. Let $(R_0,R_1,R_2,R_3)\in \mathcal A_{\rm NS}$ and let $v\in\mathcal G$  be the unique gauge transformation such that $R_0= -\dot vv^{-1}$ and $v(0)=1_G$. Let $X_0\in\gf$ satisfy $\exp X_0=u_0^{-1}(1)v(1)$. Then $p(t)={\rm e}^{-tX_0}u_0(t)v(t)^{-1}$ is an element of $\mathcal G_{00}$ and $p^{-1}.\mathcal R\in S_{\mathcal T}$. It is easy to check that both the map \eqref{slice} and its inverse are smooth, and hence $\mathcal A_{\rm NS}/\mathcal G_{00}$ is a smooth Banach manifold.
The diffeomorphism with $G\times \gf\times\gf\times\gf$ is given by 
$$\Psi: \mathcal M\to G\times \gf\times\gf\times\gf, \qquad \mathcal T\mapsto (u_0(1), T_1(0), T_2(0),T_3(0)),$$
where  $ u_0$ is the unique gauge transformation such that $T_0= -\dot u_0 u_0^{-1}$
and $u_0(0)=1_G$. 
The inverse is given by
$$\Phi: G\times \gf\times\gf\times\gf\to \mathcal M, \qquad (\gamma, \xi_1,\xi_2,\xi_3)\mapsto u_\gamma.\mathcal T_\xi,$$
where $u_\gamma\in\mathcal G$ is an arbitrary gauge transformation such that $u_\gamma(0) = 1_G$ and $ u_\gamma(1) = \gamma\in G$, and $\mathcal T_\xi = (0,T_1,T_2,T_3)$ with $(T_1,T_2,T_3)$ being the unique solution of the reduced Nahm--Schmid equations (\ref{NahmSchmidred}) with initial conditions $T_i(0) = \xi_i$. We note that $u_\gamma$ is unique up to multiplication by an element of $\mathcal G_{00}$. 
\end{proof}

\begin{remark} It follows from the proof that the tangent space to the space of solutions $\mathcal A_{NS}$ at a solution $\mathcal T$ splits as $V_{\mathcal T}\oplus H_{\mathcal T}$, where $V_{\mathcal T}$ is the tangent space to the $\mathcal G_{00}$-orbit through $\mathcal T$, and $H_{\mathcal T}$ consists of quadruples $(X_0,X_1,X_2,X_3)$ which solve the linearised Nahm equations with $X_0=u_0\xi u_0^{-1}$ for some $\xi \in\gf$. In other words, $X_0$ is an arbitrary solution of $\dot X_0+[T_0,X_0]=0$. Thus the tangent space ${\rm T}_{\mathcal T} \mathcal M$ can be identified with the set of  solutions $(X_0,X_1,X_2,X_3)$ to the system of linear equations
\begin{equation}\label{tangent}\begin{array}{l}
\dot X_1 + [T_0,X_1] + [X_0, T_1] + [T_2,X_3] + [X_2,T_3] = 0 \\
\dot X_2 + [T_0,X_2] + [X_0, T_2] - [T_3,X_1] - [X_3,T_1] = 0 \\
\dot X_3 + [T_0,X_3] + [X_0, T_3] - [T_1,X_2] - [X_1,T_2] = 0  \\
\dot X_0 + [T_0,X_0] =0.
\end{array}\end{equation}
\end{remark}

\subsection{Group actions\label{actions}}

We now want to examine various symmetries of the construction (\ref{modulispace}).

Since $\mathcal G_{00}\subset \mathcal G$ is a normal subgroup, we obtain an action of $ \mathcal G/\mathcal G_{00} \cong G\times G$ on $\mathcal M$ given by gauge transformations with arbitrary values at $t=0$ and $t=1$. Explicitly, $(u_1,u_2)\in G\times G$ acts by a gauge transformation $u\in\mathcal G$ such that $u_1=u(0), u_2=u(1)$. Any two choices of such a $u$ differ by an element of $\mathcal G_{00}$, which acts trivially on $\mathcal M$.

Furthermore, we have an action of the Lorentz group $\mathrm{SO}(1,2)$ on the moduli space. An element $A=[A_{ij}]_{i,j=1}^3 \in \mathrm{SO}(1,2)$ acts on $(T_0,T_1,T_2,T_3)\in \gf\otimes(\R\oplus\R^{1,2})$ naturally by
$$A.(T_0,T_1,T_2,T_3) = (T_0, \sum_{j=1}^3 A_{1j}T_j,\sum_{j=1}^3 A_{2j}T_j,\sum_{j=1}^3 A_{3j}T_j).$$
It is easy  to check by direct computation that this action preserves the equations (\ref{NahmSchmidUNred}) and commutes with the action of $\mathcal G$. Hence it descends to an  $\mathrm{SO}(1,2)$-action on $\mathcal M$ that commutes with the $(G\times G)$-action described above. The following result follows straightforwardly from examining the proof of Theorem \ref{MGg3}; see \cite{DancerSwann:1996} for an analogous result on the corresponding Nahm moduli space. 

\begin{proposition}
The bijection $\Psi: \mathcal M\to G\times \gf\times\gf\times\gf$ from Theorem \ref{MGg3} is compatible with the actions of $\mathrm{SO}(1,2)$ and $G\times G$ in the following way:
\begin{itemize}
\item[a)] The action of $G\times G$ on $\mathcal M$ and the action 
$$(\gamma ,\xi_1,\xi_2,\xi_3)\mapsto (u_2\gamma u_1^{-1}, u_1\xi_1u_1^{-1}, u_1\xi_2u_1^{-1}, u_1\xi_3u_1^{-1}), \quad (u_1,u_2)\in G\times G$$
on $G\times \gf\times\gf\times\gf$ are intertwined by $\Psi$.
\item[b)] The action of $\mathrm{SO}(1,2)$ on $\mathcal M$ and the action 
$$(\gamma ,\xi_1,\xi_2,\xi_3)\mapsto (\gamma, \sum_{j=1}^3 A_{1j}\xi_j,\sum_{j=1}^3 A_{2j}\xi_j,\sum_{j=1}^3 A_{3j}\xi_j)$$
on $G\times \gf\times\gf\times\gf$ are intertwined by $\Psi$. 
\end{itemize}
\end{proposition}

\subsection{Explicit solutions for $G=\mathrm{SU}(2)$\label{explicit}} 

Here we will exhibit an example of nontrivial $\su(2)$-valued solutions, which can be found using an Ansatz analogous to the one used to
calculate Nahm data for centred ${\rm SU}(2)$ magnetic 2-monopoles in Euclidean $\mathbb{R}^3$ (cf. (8.155) in \cite{ManSut}). It will turn out that this Ansatz essentially yields all solutions in the $\su(2)$-case. Setting
\begin{equation} \label{PauliAnsatz}
T_0=0,\qquad T_j(t) = f_j(t)e_j,\quad j=1,2,3,
\end{equation}
where $\{e_j \}_{j=1}^3$ is a standard basis for $\su(2)$, i.e. $[e_i,e_j] = e_k$ if $(ijk)$ is a cyclic permutation of $(123)$, then the Nahm--Schmid equations yield the following system of ODEs for the
functions $f_j$:
\begin{eqnarray*}
 \dot f_1&=&-f_2 f_3,\\
 \dot f_2&=& f_3 f_1,\\
 \dot f_3&=& f_1 f_2.
\end{eqnarray*}
The general solution to this system can be expressed in terms of Jacobi elliptic functions~\cite{ByrFriHEI}
with arbitrary modulus $\kappa \in [0,1]$:
\begin{eqnarray*}
f_1(t)&=& a \kappa\,  {\rm sn}_\kappa(a  t + b),\\
f_2(t)&=&a \kappa   \, {\rm cn}_\kappa(a t+ b),\\
f_3(t)&=& -a \, {\rm dn}_\kappa(a t+ b),
\end{eqnarray*}
where $a, b\in \R$ are arbitrary constants. For $\kappa = 0$ this yields with $\mathrm{dn}_0 \equiv 1$  the trivial solution $T_3 = \mathrm{constant}, T_1=T_2 = 0$. For $\kappa  = 1$ we obtain 
$\mathrm{sn}_1(t) = \tanh(t), \mathrm{cn}_1(t) = \mathrm{dn}_1(t) = {\rm sech}(t)$.

The solution corresponding to fixed $\kappa, a,b$ is smooth for all $t \in \mathbb{R}$, as we expected from Corollary~\ref{GlobalReg}. For $\kappa\in (0,1)$ it is also periodic with period $4 {\rm K}(\kappa)/a$, where $\rm K$ is the complete elliptic integral of the first kind~\cite{ByrFriHEI}. If we fix an invariant inner product on $\su(2)$ such that the basis $\{e_1,e_2,e_3\}$ is orthonormal, then we see that the conserved quantities $\|T_1\|^2 + \|T_i\|^2$ are  
\begin{equation} \label{rels} f_1^2 + f_2^2 = a^2 \kappa^2, \qquad f_1^2 + f_3^2 = a^2.\end{equation}
Note that for $\kappa =1$ the conserved quantities coincide, and the solution is of course not periodic in this case. 

Matsoukas has shown in his DPhil thesis~\cite{Matsoukas:2010}  that in fact any solution to the Nahm--Schmid equations with values in $\su(2)$ may be put into this form. To see this, first gauge $T_0$ away  by a gauge transformation in $G\cong \{1_G\}\times G\subset G\times G$, i.e. a gauge transformation which equals the identity at $t=0$. Then use the $\mathrm{SO}(1,2)$-action to put the solution $(T_1,T_2,T_3)$ into standard form such that $\langle T_i,T_j\rangle = 0$ whenever $i\neq j$. Now the vector space $\su(2)$ is three-dimensional, so it follows that for generic $t\in \R$ the elements $T_1(t),T_2(t),T_3(t)$ form an orthogonal basis. The Nahm--Schmid equations imply that $T_i(t)$ and $\dot T_i(t) = \pm[T_j(t),T_k(t)]$ are linearly dependent for each $t$. It follows that the direction of $T_i$ does not vary with time for each $i$. This implies  that there exists an orthonormal basis of $\su(2)$, say the standard one used above, such that the $T_j$ are of the form 
$$T_j(t) = f_j(t)e_j$$
and hence reduce to the Ansatz used above to obtain explicit solutions.

We summarise this discussion as follows.

\begin{proposition}[cf. \cite{Matsoukas:2010}] Let $\mathcal M$ be the moduli space of $\su(2)$-valued solutions to the Nahm--Schmid equations on $[0,1]$ modulo the gauge group $\mathcal G_{00}$. Let $\mathcal{T}\in \mathcal M$. Then the $\mathrm{SO}(1,2)\times G$-orbit of $\mathcal T$ contains an $\mathrm{SO}(3)\cap \mathrm{SO}(1,2)\cong \mathrm{SO}(2)$-orbit of solutions of the form (\ref{PauliAnsatz}),
with respect to a standard orthonormal basis of $\su(2)$. 
\end{proposition}

\subsection{Hypersymplectic interpretation\label{hsinterpret}}
At a point $\mathcal T=(T_0,T_1,T_2,T_3)\in\mathcal A$, the tangent space to $\mathcal A$ has the description 
$${\rm T}_\mathcal T\mathcal A = C^1(U,\gf)\otimes \R^4.$$ 

On the $\R^4\cong \mathbb B$ factor, we have the split-quaternionic structure $I,S,T$ induced  from multiplication by $\rm -i,s,t$  on $\mathbb B$ from the right. Explicitly:
\begin{eqnarray}\label{I}
I(X_0,X_1,X_2,X_3) &=& (X_1,-X_0,-X_3,X_2)\\
\label{S} S(X_0,X_1,X_2,X_3) &=& (X_2,X_3,X_0,X_1)\\
\label{T} T(X_0,X_1,X_2,X_3) &=& (X_3,-X_2,-X_1,X_0).
\end{eqnarray}

We endow $\mathcal A = C^1(U,\gf)\otimes \R^{2,2}$ with the indefinite metric given by the tensor product of the $L^2$-metric and the standard metric on $\R^{2,2}$. On tangent vectors $X=(X_0,X_1,X_2,X_3)$ and $Y=(Y_0,Y_1,Y_2,Y_3)$,  we have the formula
\begin{equation}\label{Metric}
g(X,Y) = \int_U \sum_{i=0}^3 \eta_{ii}\langle X_i(t),Y_i(t)\rangle\,  \mathrm dt.
\end{equation}
Here we wrote $\eta_{ii}$ for the diagonal coefficients of the metric on $\R^{2,2}$, i.e. $\eta_{00} = \eta_{11} = 1, \eta_{22}=\eta_{33} = -1$.
This gives $\mathcal A$ the structure of a flat infinite-dimensional hypersymplectic manifold with symplectic forms 
$$\omega_I(\cdot,\cdot) = g(I \cdot, \cdot),\quad \omega_S(\cdot,\cdot) = g(S \cdot,\cdot),\quad \omega_T(\cdot,\cdot) = g(T\cdot,\cdot).$$ 
The action of the group of gauge transformations $\mathcal G$ preserves this flat hypersymplectic structure.

By calculating $\frac{\mathrm d}{\mathrm d\theta}|_{\theta =0}\exp(\theta\xi).(T_0,T_1,T_2,T_3)$ for $\xi\in \mathrm{Lie}(\mathcal G)$, we see that the fundamental vector fields associated to the action of $\mathcal G$ (and hence also for $\mathcal G_{00}$) at a point $\mathcal T=(T_0,T_1,T_2,T_3)\in \mathcal A$ are given by
$$X^\xi_\mathcal T = (-\dot \xi + [\xi,T_0], [\xi,T_{1}], [\xi,T_{2}], [\xi,T_{3}]).$$

\begin{proposition}\label{mmaps}
The action of the group $\mathcal G_{00}$ on $\mathcal A$ preserves the hypersymplectic structure and the hypersymplectic moment map at $\mathcal T = (T_0,T_1,T_2,T_3)\in\mathcal A$ is given by
\begin{eqnarray*}
\mu_I(\mathcal T) &=& -\dot T_1 - [T_0,T_1] - [T_2,T_3], \\
\mu_S(\mathcal T) &=& \dot T_2 +[T_0,T_2] - [T_3,T_1], \\
\mu_T(\mathcal T) &=& \dot T_3 + [T_0,T_3] - [T_1,T_2].
\end{eqnarray*}
Hence we can write the moduli space $\mathcal M$ of solutions to the Nahm--Schmid equations formally as the hypersymplectic quotient:
$$\mathcal M = \mu^{-1}(0)/\mathcal G_{00}.$$
\end{proposition}
\begin{proof}
We only exhibit the calculation for $\mu_I$, as the other two equations are obtained analogously.
Let $\mathcal T\in\mathcal A$, $\xi\in\mathrm{Lie}(\mathcal G_{00})$. First observe that 
$$IX^\xi_\mathcal T = I(-\dot\xi + [\xi,T_0],[\xi,T_{1}], [\xi,T_{2}], [\xi,T_{3}]) = ([\xi,T_1], \dot\xi-[\xi,T_0],-[\xi,T_3], [\xi,T_2]).$$
Thus, we can calculate for $Y\in {\rm T}_\mathcal T\mathcal A$, using integration by parts and the boundary condition $\xi(0) = 0 = \xi(1)$:
\begin{eqnarray*}
\omega_I(X^\xi,Y) &=& g(IX_\mathcal T^\xi,Y)\\
&=& \int_0^1\langle [\xi,T_1], Y_0\rangle +\langle \dot\xi-[\xi,T_0],Y_1\rangle +\langle [\xi,T_3], Y_2\rangle -\langle [\xi,T_2], Y_3\rangle \mathrm dt\\
&=& \langle \xi,Y_1\rangle|_0^1 + \int_0^1\langle \xi, -\dot\xi - [T_0,Y_1] - [Y_0,T_1] - [T_2,Y_3]-[Y_2,T_3]\rangle\mathrm dt \\
&=&  \int_0^1\langle \xi, -\dot\xi - [T_0,Y_1] - [Y_0,T_1] - [T_2,Y_3]-[Y_2,T_3]\rangle\mathrm dt.
\end{eqnarray*}
The assertion follows.
\end{proof}

\begin{lemma}\label{Lem:Slice}
Let $\mathcal T\in \mathcal A$. 
The orthogonal complement to the tangent space to the $\mathcal G_{00}$-orbit through $\mathcal T$ with respect to the indefinite metric $g$ is given by tangent vectors $(X_0,X_1,X_2,X_3)\in {\rm T}_\mathcal T\mathcal A$ satisfying the equation
\begin{equation} \label{floweq}
\dot X_0 + \sum_{i=0}^3 \eta_{ii}[T_i,X_i] = 0.
\end{equation}
\end{lemma}
\begin{proof}
This is a standard calculation using integration by parts and the boundary conditions defining $\mathrm{Lie}(\mathcal G_{00})$. Let $\xi\in\mathrm{Lie}(\mathcal G_{00})$, $(X_0,X_1,X_2,X_3)\in {\rm T}_\mathcal T\mathcal A$.  Then
\begin{eqnarray*}
g(X^\xi_\mathcal T, X) &=& \int_0^1\langle -\dot\xi + [\xi, T_0], X_0\rangle + \sum_{i=1}^3\eta_{ii}\langle [\xi, T_i],X_i\rangle\mathrm dt\\
&=&  \int_0^1\langle \xi , \dot X_0\rangle + \sum_{i=0}^3\eta_{ii}\langle \xi,[ T_i,X_i]\rangle \mathrm dt.
\end{eqnarray*}
The result follows.
\end{proof}

\begin{definition}
Consider the Banach manifold $\mu^{-1}(0)\subset \mathcal A$ of solutions to the Nahm--Schmid equations. The \emph{degeneracy locus $\mathcal D$} is  the set of points $\mathcal T \in \mu^{-1}(0)$ such that the metric $g$ restricted to the tangent space of the $\mathcal G_{00}$-orbit through $\mathcal T$ is degenerate. 
\end{definition}

\begin{lemma}
The actions of the gauge group $\mathcal G$ and of $\mathrm{SO}(1,2)$ both  preserve the degeneracy locus $\mathcal D$. 
\end{lemma}
\begin{proof}
This follows immediately from the fact that both $\mathcal G$ and $\mathrm{SO}(1,2)$ act by isometries. 
\end{proof}

\begin{proposition}[\cite{Matsoukas:2010}] \label{previousprop}
The degeneracy locus $\mathcal D$ consists of exactly those solutions to the Nahm--Schmid equations such that the boundary value problem 
$$\frac{\mathrm d^2\xi}{\mathrm dt^2} + [\dot T_0, \xi] + 2[T_0, \dot \xi] + \sum_{i=0}^3 \eta_{ii}[T_i,[T_i,\xi]] = 0,\quad \xi\in \mathrm{Lie}(\mathcal G_{00})$$
has a nontrivial solution. 
\end{proposition}
\begin{proof}
The tangent space to the orbit through $\mathcal T$ is spanned by the values of the fundamental vector fields $X^\xi_\mathcal T$. We thus need to find $\xi\in\mathrm{Lie}(\mathcal G_{00})$ such that the tangent vector 
$X^\xi_\mathcal T\in {\rm T}_\mathcal T\mathcal A$ satisfies  equation (\ref{floweq}). 
Plugging $$X^\xi_\mathcal T = (-\dot\xi + [\xi, T_0], [\xi, T_1], [\xi, T_2], [\xi, T_3])$$ 
into that equation yields the desired result. 
\end{proof}

\begin{proposition}
The complement $\mu^{-1}(0)\setminus \mathcal D$ of the degeneracy locus consists exactly of those solutions $\mathcal T\in \mu^{-1}(0)$ for which the linear operator 
\begin{equation} \label{linearop}
\Delta_\mathcal T: \mathrm{Lie}(\mathcal G_{00})\to C^0(U,\gf), \quad \xi\mapsto\frac{\mathrm d^2\xi}{\mathrm dt^2} + [\dot T_0, \xi] + 2[T_0, \dot \xi] + \sum_{i=0}^3 \eta_{ii}[T_i,[T_i,\xi]]
\end{equation}
is an isomorphism. In particular, $\mu^{-1}(0)\setminus \mathcal D$ is open.
\end{proposition}
\begin{proof}
Owing to the Arzel\`a--Ascoli theorem,  the operator $\Delta_\mathcal T$ is  a compact perturbation of the operator $\frac{\mathrm d^2}{\mathrm dt^2}:\mathrm{Lie}(\mathcal G_{00})\to C^0(U,\gf)$, which is an isomorphism. Hence $\Delta_\mathcal T$ is Fredholm  of index zero. Thanks to Proposition~\ref{previousprop}, $\Delta_\mathcal T$ is injective if $\mathcal T$ is not contained in the degeneracy locus. Hence it must be an isomorphism. 
\end{proof}

An analogous argument to the one outlined in \cite[pp. 4--5]{Kronheimer:1988} (see also \cite{DancerSwann:1996}) for the Nahm case yields then the following result.

\begin{theorem} The submanifold
$$\mathcal M^0 := \left(\mu^{-1}(0)\setminus \mathcal D\right)/\mathcal G_{00}$$
of $\mathcal M$ 
is a smooth hypersymplectic manifold.
The tangent space ${\rm T}_{\mathcal T} \mathcal M^0$ can be identified with the set of  solutions $X$ to the system of linear equations (cf. \eqref{tangent})
\begin{eqnarray*}
\dot X_1 + [T_0,X_1] + [X_0, T_1] + [T_2,X_3] + [X_2,T_3] &=& 0 \\
\dot X_2 + [T_0,X_2] + [X_0, T_2] - [T_3,X_1] - [X_3,T_1] &=& 0 \\
\dot X_3 + [T_0,X_3] + [X_0, T_3] - [T_1,X_2] - [X_1,T_2] &=& 0 \\
\dot X_0 + [T_0,X_0] + [T_1,X_1] -[T_2,X_2] -[T_3,X_3] &=& 0,
\end{eqnarray*}
and the hypersymplectic structure is given by \eqref{I}--\eqref{T}.
\end{theorem}
The last equation above says that $X$ is orthogonal to the $\mathcal G_{00}$-orbit, as we know from Lemma \ref{Lem:Slice}. 

We can obtain some more quantitative information about the degeneracy locus. 

\begin{proposition}\label{Prop:DegenLoc}
Let $\mathcal T \in \mu^{-1}(0)$ be a solution such that $2\sup(\|T_2(t)\|^2+\|T_3(t)\|^2) <\pi^2$. Then $\mathcal T$ does not belong to the degeneracy locus $\mathcal D$. This holds in particular for any solution with $T_2 \equiv 0 \equiv T_3$. 
\end{proposition}
\begin{proof}
Since the degeneracy locus is invariant under the action of $\mathcal G$, we may restrict attention to solutions such that $T_0\equiv 0$. A solution $\mathcal T$ lies in the degeneracy locus if and only if there is a nontrivial solution $\xi\in\mathrm{Lie}(\mathcal G_{00})$ to the linear boundary value problem 
\begin{equation}\label{degen}\ddot \xi + \left(\mathrm{ad}(T_1)^2-\mathrm{ad}(T_2)^2-\mathrm{ad}(T_3)^2\right)(\xi) =0, \qquad \xi(0) = 0 = \xi(1).\end{equation}
Let us write this ODE as $\ddot \xi + A\xi = 0$. 
Since $\mathrm{ad}(T_i)$ is skew-symmetric with respect to the invariant inner product, we have for any $\xi\in C^0([0,1],\gf)$
$$\langle A\xi,\xi\rangle = -\|[T_1,\xi]\|^2 + \|[T_2,\xi]\|^2+\|[T_3,\xi]\|^2 \leq\|[T_2,\xi]\|^2+ \|[T_3,\xi]\|^2.$$ 
Using the pointwise inequality $\|[\xi,\eta]\|^2\leq 2\|\xi\|^2\|\eta\|^2$, this implies
$$\langle A\xi,\xi\rangle\leq 2\sup(\|T_2(t)\|^2+\|T_3(t)\|^2)\|\xi\|^2.$$
Setting $M:=2\sup(\|T_2(t)\|^2+\|T_3(t)\|^2)$, we can conclude that $f(t)=\|\xi\|^2$ satisfies the differential inequality $\ddot f +Mf\geq 0$. Let
$g(t) := \frac{\dot f(0)}{\sqrt{M}}\sin(\sqrt{M}t)$ for $t\in U$, which solves $\ddot g + Mg =0$ with initial condition $g(0) =0$, $\dot g(0)= \dot f(0)$. Consider the function $\phi(t)=f(t)/g(t)$. Provided that $M <\pi^2$, $\phi$ is well-defined and  differentiable on $[0,1]$. The inequality $\ddot f +Mf\geq 0$ and the equality $\ddot g + Mg =0$ imply that  
$\ddot fg-\ddot gf\geq 0$ on $[0,1]$ and integrating this inequality shows that $\phi$ is monotonically increasing there. Since $\phi(0)=1$, $\phi$, and hence $\xi$, cannot have another zero on $(0,1]$. Thus, if $M <\pi^2$, then the boundary value problem $\ddot \xi + A\xi=0$, $\xi(0)=\xi(1)=0$  has only the trivial solution $\xi\equiv 0$. 
\end{proof}

On the other hand we have

\begin{proposition}\label{Prop:DegLoc} Let $\mathcal T=(T_0,T_1,T_2,T_3)$ be a non-constant solution of the Nahm--Schmid equations on $[0,1]$ such that, for some $A\in {\rm SO}(1,2)$, one of the components $\tilde T_i=(A\mathcal T)_i$, $i=1,2,3$, vanishes at $t=0$ and $t=1$. Then $\mathcal T$ belongs to the degeneracy locus $\mathcal D$.\label{TD} 
\end{proposition}
\begin{proof} Without loss of generality we can assume that $T_0\equiv 0$, $A=1$ and that $T_1$ is not constant and vanishes at $0$ and $1$. It is easy to check that $\xi=T_1$ satisfies \eqref{degen}.
\end{proof}

\begin{example} If $G={\rm SU}(2)$ we can say more about equation \eqref{linearop} and the degeneracy locus. For a solution in the standard form $ T_0=0,\enskip T_j(t) = f_j(t)e_j$, $j=1,2,3$, we have $(\ad T_j)^2(\xi)=-f_j^2\langle \xi, e_j\rangle e_j$ and hence equation \eqref{degen} is diagonal in the standard basis of $\su(2)$: 
\begin{eqnarray*} \ddot\xi_1 & =& -(f_2^2+f_3^2)\xi_1\\
\ddot\xi_2 & =& (f_1^2-f_3^2)\xi_2\\
\ddot\xi_3 & =& (f_1^2-f_2^2)\xi_3.
\end{eqnarray*}
where $\xi=\sum_{i=1}^3 \xi_i\sigma_i$.  Using \eqref{rels} we can rewrite these equations as
\begin{eqnarray*}\ddot\xi_1 = & (2f_1^2 -a^2-\kappa^2 a^2) &= a^2\bigl(2\kappa^2 {\rm sn}^2_\kappa(at+b)-1-\kappa^2\bigr)\xi_1 \\
\ddot\xi_2= & (2f_1^2-a^2)\xi_2&= a^2\bigl(2\kappa^2 {\rm sn}^2_\kappa(at+b) -1\bigr)\xi_2 \\ \ddot\xi_3=& (2f_1^2-\kappa^2a^2)\xi_3 &= 
a^2\kappa^2\bigl(2\, {\rm sn}^2_\kappa(at+b) -1\bigr)\xi_3.\end{eqnarray*}
These equations have particular solutions $\xi_1={\rm sn}_\kappa(at+b)$, $\xi_2={\rm cn}_\kappa(at+b)$ and $\xi_3={\rm dn}_\kappa(at+b)$ (cf. Proposition \ref{TD}). Since ${\rm dn}_\kappa$ does not vanish on the real line, the Sturm Separation Theorem \cite[Cor. XI.3.1]{Hart} implies that the third equation cannot have a nontrivial solution vanishing at two points. On the other hand ${\rm sn}_\kappa$ (resp. ${\rm cn}_\kappa$) vanishes at points $2mK$ (resp. $(2m+1)K$), $m\in\mathbb{Z}$. Thus a solution $T$ with $b\in K\mathbb{Z}$, $a\in 2K\mathbb{Z}$ belongs to the degeneracy locus $\mathcal D$, while, owing again to the Sturm Separation Theorem, any solution with $[a,a+b]$ properly contained in $[mK,(m+1)K]$ does not belong to $\mathcal D$.
\end{example}

\begin{corollary}
For any non-abelian compact Lie algebra $\gf$, the degeneracy locus $\mathcal D\subset \mathcal M$ is non-empty. 
\end{corollary}
\begin{proof}
Let $\rho:\su(2)\to\gf$ be a (non-trivial) Lie algebra homomorphism and let $(T_0,T_1,T_2,T_3)$ be an $\su(2)$-valued solution to the  Nahm--Schmid equations such that $T_1$ is not constant and vanishes at $t=0,1$. Then  $(\rho(T_0),\rho(T_1),\rho(T_2),\rho(T_3))\in\mathcal D$ by Proposition \ref{Prop:DegLoc}.
\end{proof}

\subsection{Complex structures\label{complex}}
Denote by $\gf^\cx=\gf\otimes \cx$ the complexification of $\gf$.
We may identify the complex manifold $(\mathcal A,I)$ with $C^1(U,\gf^\mathbb C))\otimes \mathbb C^2$ via $$\mathcal T = (T_0,T_1,T_2,T_3)\mapsto (\alpha,\beta), \qquad \text{where } \; \alpha := T_0-{\rm i}T_1, \beta := T_2+{\rm i}T_3.$$ Writing a tangent vector to $C^1(U,\gf^\mathbb C)\otimes \mathbb C^2$ as $(a,b)$, the holomorphic symplectic form $\omega_I^\mathbb C = \omega_S+{\rm i}\omega_T$ is given by
$$\omega_I^\mathbb C((a_1,b_1),(a_2,b_2)) = \int_0^1\langle b_1,a_2\rangle - \langle a_1,b_2\rangle\mathrm dt.$$
The vanishing of the hypersymplectic moment map is the same as saying that the complex equation $\mu^\mathbb C:=\mu_S+{\rm i}\mu_T = 0$ and the real equation $\mu_I =0$ are satisfied simultaneously. In the case of the Nahm--Schmid equations, this gives respectively
\begin{eqnarray}\label{cx}
 \dot\beta + [\alpha,\beta] &=& 0, \\
 \label{re}\dot\alpha + \dot\alpha^\ast  + [\alpha,\alpha^\ast ] - [\beta,\beta^\ast ] &=& 0.
\end{eqnarray}

Let $G^\cx$ be the complexification of $G$, i.e. $G^\cx$ is a complex Lie group with Lie algebra $\gf^\cx$ and maximal compact subgroup $G$.
The complex equation is invariant under the action of the \emph{complexified gauge group} 
$$\mathcal G^\mathbb C = C^2(U,G^\cx)$$
and its normal subgroup
$$\mathcal G_{00}^\mathbb C := \{u\in\mathcal  G^\mathbb C) \ |\ u(0) = \mathbf{1}_{G^\cx} = u(1)\}$$
acting by
$$u.\alpha = u\alpha u^{-1}  -\dot u u^{-1}, \qquad u.\beta = u\beta u^{-1}.$$
A computation similar to the one in Proposition \ref{mmaps} shows that the complex equation is the vanishing condition for the moment map of the action of $\mathcal G_{00}^\mathbb C$ with respect to $\omega_I^\mathbb C$.

We write the moduli space of solutions to the complex equation modulo $\mathcal G_{00}^\mathbb C$ as
$$\mathcal N := \{(\alpha,\beta): U\to \gf^\C\times \gf^\C \ |\ \dot \beta + [\alpha, \beta]  =0\}/\mathcal G_{00}^\C.$$
One can show that this is a smooth Banach manifold.
Just like in the proof of Theorem \ref{MGg3} we obtain a diffeomorphism
\begin{eqnarray*}
\Phi: \mathcal N &\to& G^\cx \times \gf^\C \cong {\rm T}^\ast G^\C, \\
(\alpha,\beta)  &\mapsto& (u_0(1), \beta(0)),
\end{eqnarray*}
where $u_0\in\mathcal G^\mathbb C$ is the unique complex gauge transformation with $\alpha = -\dot u_0u_0^{-1}$ such that $u_0(0) = 1_{G^\mathbb C}$. The space $\mathcal N$ is in a natural way a complex symplectic quotient, and it can be checked that $\Phi$ is holomorphic and pulls back the canonical symplectic form on ${\rm T}^\ast G^\mathbb C$  to the symplectic form $\omega_I^\mathbb C$.

We obviously have a natural map $\mathcal M\to \mathcal N$, since any solution to the Nahm--Schmid equations gives a solution to the complex equation. 
The question to ask at this point is the following: Given a solution to the complex equation $\dot \beta+ [\alpha, \beta] = 0$, does its $\mathcal G_{00}^\C$-orbit contain a solution to the real equation, and to what extent is this solution unique?

Using work of Donaldson \cite{Donaldson:1984a}, Kronheimer answers this question affirmatively for the usual Nahm equations by showing that every $\mathcal G_{00}^\mathbb C$-orbit of a solution to the complex equation contains a unique $\mathcal G_{00}$-orbit of solutions to the real equation. In the hyperk\"ahler case, it thus turns out that the map $\mathcal{M}\rightarrow \mathcal{N}$ is a bijection, inducing therefore a hyperk\"ahler structure on ${\rm T}^\ast G^\C$. 

We think of $\gf$ as a subalgebra of $\un(n)$ for some $n\in\mathbb N$. Let us write the real equation as $\mu_I(\alpha, \beta) = 0$. We now aim at describing how this equation behaves under complex gauge transformations. 

To simplify calculations, we define operators on $C^\infty([0,1], \gf^\C)$ by  
\begin{eqnarray*}
\bar\partial_\alpha = \frac{\mathrm d}{\mathrm dt} + [\alpha,\cdot], &\quad &
\partial_\alpha =\frac{\mathrm d}{\mathrm dt} - [\alpha^\ast ,\cdot],\\
\bar\partial_\beta =  [\beta,\cdot], &\quad &
\partial_\beta = -[\beta^\ast ,\cdot].
\end{eqnarray*}
The following two lemmas are then obtained by straightforward calculations.
\begin{lemma}
$$\mu_I(\alpha,\beta) = [\partial_\alpha,\bar\partial_\alpha] -[\partial_\beta, \bar\partial_\beta] ,$$
as operators on $C^\infty([0,1],\mathbb C^n)$.
\end{lemma}

\begin{lemma}
Let $u\in \mathcal G^\mathbb C$ be a complex gauge transformation. Then
\begin{eqnarray*}
\bar\partial_{u.\alpha} = u\circ \bar\partial_\alpha \circ u^{-1} & \quad &
\partial_{u.\alpha} = (u^\ast)^{-1} \circ \partial_\alpha \circ u^\ast \\
\bar\partial_{u.\beta} = u\circ \bar\partial_\beta \circ u^{-1} &\quad &
\partial_{u.\beta} = (u^\ast)^{-1} \circ \partial_\beta \circ u^\ast.
\end{eqnarray*}
\end{lemma} 

Using this description of the real moment map, it is now easy to work out the behaviour of the real equation under complex gauge transformations. 
\begin{lemma}
Let $u \in \mathcal G^\mathbb C$ be a complex gauge transformation, and put $h := u^\ast u$. Then
$$u^{-1}\left(\mu_I(u.\alpha, u.\beta)\right)u = \mu_I(\alpha, \beta) - \bar\partial_\alpha( h^{-1}(\partial_\alpha h)) + \bar\partial_\beta( h^{-1}(\partial_\beta h)). $$
\end{lemma}

Now given a solution $(\alpha, \beta)$ to the complex equation, we want to find a self-adjoint and positive $h\in\mathcal G_{00}^\C$  solving the boundary value problem 
\begin{equation}\label{h-RealEq}\mu_I(\alpha, \beta) - \bar\partial_\alpha(h^{-1}(\partial_\alpha h)) + \bar\partial_\beta(h^{-1}(\partial_\beta h)) = 0,  \qquad h(0) = 1 = h(1).
\end{equation}
Existence and uniqueness will then imply that the complex gauge transformation $u = h^{1/2}$ takes $(\alpha,\beta)$ to a solution of the real equation. 

\begin{lemma}
Let $\mathcal T = (T_0,T_1,T_2,T_3)$ be a  solution to the Nahm--Schmid equations, i.e. a solution to the complex equation with $\mu_I(\alpha, \beta) = 0$.  Then the linearisation of the boundary value problem
$$- \bar\partial_\alpha(h^{-1}(\partial_\alpha h)) + \bar\partial_\beta(h^{-1}(\partial_\beta h)) = 0$$
is given by the operator $-\Delta_\mathcal T$ in \eqref{linearop}. 
\end{lemma}
\begin{proof}
Suppose that we have a one-parameter family $h: (-\epsilon,\epsilon)\times [0,1]\to G^\C$ of self-adjoint solutions
$$h(s,t) = \exp({\rm i}\xi(s,t)), \qquad h(0,t)= 1,\enskip h(s,0) = 1 = h(s,1)$$
to the boundary value problem (\ref{h-RealEq}), where $\xi(s,t) (-\epsilon,\epsilon)\times [0,1]\to \gf$. 
Write $\xi'$ for the partial derivative of $\xi$ with respect to $s$ at $s=0$. Note that the condition $h(s,0) = 1 = h(s,1)$ implies that $\xi'(0) = 0 = \xi'(1)$, i.e. $\xi'\in \mathrm{Lie}(\mathcal G_{00})$. We compute the linearisation of (\ref{h-RealEq}) for $(\alpha, \beta)$ with $\mu_I(\alpha,\beta) =0$ and denote the linear operator obtained in this way by $L: \mathrm{Lie}(\mathcal G_{00})\to C^0([0,1],\gf)$. Explicitly, 
\[
L\xi' = \left. -{\rm i}\frac{\mathrm d}{\mathrm ds}\right|_{s=0}\left(- \bar\partial_\alpha(h^{-1}(\partial_\alpha h)) + \bar\partial_\beta(h^{-1}(\partial_\beta h))\right).
\]
Using the equation $0=\mu_I(\alpha,\beta) = \dot T_1 + [T_0,T_1]+ [T_2,T_3]$ and the Jacobi identity we see
\begin{eqnarray*}
L\xi'  &=& \left.-{\rm i} \frac{\mathrm d}{\mathrm ds}\right|_{s=0}\left(- \bar\partial_\alpha(h^{-1}(\partial_\alpha h)) + \bar\partial_\beta(h^{-1}(\partial_\beta h))\right)\\
&=& - \bar\partial_\alpha(\partial_\alpha \xi') + \bar\partial_\beta(\partial_\beta \xi')\\
&=& -\left(\frac{\mathrm d}{\mathrm dt} + [\alpha,\cdot]\right)\left(\dot\xi' - [\alpha^*,\xi']\right) - [\beta,[\beta^*,\xi]]\\
&=&  -\ddot\xi' + [\dot\alpha^\ast , \xi'] + [\alpha^\ast ,\dot \xi']  - [\alpha,\dot \xi']  +[\alpha, [\alpha^\ast ,\xi']] - [\beta,[\beta^\ast ,\xi']] \\
&=& -\ddot\xi' - [\dot T_0, \xi'] -{\rm i}[\dot T_1,\xi'] - [T_0 +{\rm i}T_1,\dot\xi']  - [T_0-{\rm i}T_1,\dot\xi']  \\
& &  - [T_0-{\rm i}T_1, [T_0+{\rm i}T_1,\xi']] + [T_2+{\rm i}T_3,[T_2-{\rm i}T_3,\xi']] \\
&=& -\ddot\xi' - [\dot T_0, \xi'] +{\rm i}[[T_0,T_1] + [T_2,T_3],\xi'] - [T_0 +{\rm i}T_1,\dot\xi']  - [T_0-{\rm i}T_1,\dot\xi']  \\
& &  - [T_0-{\rm i}T_1, [T_0+{\rm i}T_1,\xi']] + [T_2+{\rm i}T_3,[T_2-{\rm i}T_3,\xi']] \\
&=& -\ddot\xi' - 2[T_0,\dot \xi'] - [\dot T_0, \xi']  - [T_0, [T_0, \xi']] - [T_1,[T_1,\xi']]  + [T_2,[T_2,\xi']] \\
&& + [T_3,[T_3,\xi']]\\
&=& -\Delta_\mathcal T \xi'.
\end{eqnarray*}
\end{proof}

\begin{theorem}
Let $(\alpha,\beta)$ be a solution to the Nahm--Schmid equations, and assume that $(\alpha,\beta)$ is not contained in the degeneracy locus $\mathcal D$. Then there exists a constant $\epsilon >0$ such that for any solution $(\tilde\alpha,\tilde\beta)$ of the complex equation with $\|\tilde\alpha- \alpha\|_{C^1} + \|\tilde\beta -\beta\|_{C^1} <\epsilon$ there exists a unique $u\in\mathcal G_{00}^\C$ close to the identity which is self-adjoint such that $u.(\tilde\alpha,\tilde\beta)$ solves the real equation. 
\end{theorem}
\begin{proof}
Consider the map 
$$\mathcal F: \mathrm{Lie}(\mathcal G_{00})\times \mathcal A \to C^0(U,\gf)$$
given by 
$$\mathcal F(\xi,(\tilde\alpha,\tilde\beta)) = \mu_I(\tilde\alpha,\tilde\beta)- \bar\partial_\alpha({\rm e}^{{\rm i}\xi}(\partial_\alpha {\rm e}^{-{\rm i}\xi})) + \bar\partial_\beta({\rm e}^{{\rm i}\xi}(\partial_\beta {\rm e}^{-{\rm i}\xi})),$$
which satisfies $\mathcal F(0,(\alpha,\beta)) = 0$. 
The partial derivative of $\mathcal F$ at the point $(0,(\alpha,\beta))$ in the $\mathrm{Lie}(\mathcal G_{00})$-direction is the operator $\Delta_\mathcal T$. Owing to the assumption, this is an isomorphism. Hence the assertion follows from the implicit function theorem. 
\end{proof}

In particular, we note that this result applies to solutions $\mathcal T$ such that $$2\sup(\|T_2\|^2+\|T_3\|^2) <\pi^2,$$ owing to Proposition \ref{Prop:DegenLoc}.

\begin{corollary}
Consider the complex symplectic manifold $(\mathcal M^0,I,\omega_I^\C)$. Then
there exists an open neighbourhood of the set $\{\mathcal T\in\mathcal M\ | \ T_2=T_3=0\}$ which is isomorphic to a suitable open neighbourhood of the zero section in $\mathrm T^*G^\C$ with its canonical complex symplectic structure.
\end{corollary}

\subsection{Product structures\label{product}}
In this subsection, we investigate the paracomplex structure~\cite[Ch.~V]{Lieb} of the moduli space $\mathcal M$. 

As a first step, we identify the  space of field configurations $\mathcal A$, equipped with the product structure $S$, as a parak\"ahler manifold~\cite{Lieb}. Let 
$$\mathcal B := \left\{ \left. \frac{\mathrm d}{\mathrm dt} + A \ \right| A\in C^1( U, \mathfrak \gf) \ \right\}$$
be the space of $G$-connections on the interval $U$. Then we can naturally identify
\begin{equation} \label{cotanB}
{\rm T}^\ast \mathcal B = \left\{\left. \left(\frac{\mathrm d}{\mathrm dt} + A, B\right)  \right| A, B \in C^1( U, \mathfrak \gf) \ \right\} \cong C^1(U,\mathfrak{g})^2.
\end{equation}

\begin{proposition}\label{prop:Prod}
The map $P: (\mathcal A, S) \to (T^\ast \mathcal B\times T^\ast \mathcal B, {\rm id} \oplus (-{\rm id}))$ given by 
\begin{equation} \label{mapP}
\left(T_0, T_1,T_2,T_3\right) \mapsto \left(\left(T_0 + T_2, T_1+T_3\right),\left(T_0 - T_2, T_1-T_3\right)\right)
\end{equation}
is a diffeomorphism respecting the product structures.
\end{proposition}
\begin{proof}
Clearly, the map $P$ is a diffeomorphism.
It is also straightforward to check by direct calculation that 
$$\mathrm dP \circ S = (1\oplus(-1)) \circ\mathrm dP.$$
\end{proof}

Now $P$ is not just a diffeomorphism. We observe that it intertwines the symplectic form $\omega_I$ and the product of canonical symplectic forms on ${\rm T}^\ast \mathcal B \times {\rm T}^\ast \mathcal B$. 

\begin{proposition}
Consider the symplectic form $\Omega := \omega_{T^\ast \mathcal B} \oplus \omega_{T^\ast \mathcal B}$  on ${\rm T}^\ast \mathcal B \times {\rm T}^\ast \mathcal B$. Then 
$$\omega_I = P^\ast \Omega.$$
\end{proposition}
\begin{proof}
The symplectic form $\omega_I$ is given at an element $\mathcal{T}\in \mathcal{A}$ by 
$$\omega_I(X,Y) = -g(X_0,Y_1) + g(X_1,Y_0) - g(X_2,Y_3) + g(X_3,Y_2),$$
where $X =(X_0,X_1,X_2,X_3) , Y=(Y_0,Y_1,Y_2,Y_3) \in {\rm T}_\mathcal{T}\mathcal A$ are two tangent vectors at $\mathcal T$. 
The symplectic form $\Omega$ is given by
$$\Omega (V,W) = g(V_0,W_1) - g(V_1,W_0) + g(V_2,W_3) - g(V_3,W_2),$$
with $V = (V_0,V_1,V_2,V_3), W = (W_0,W_1,W_2,W_3) \in {\rm T}({\rm T}^\ast \mathcal B \times {\rm T}^\ast \mathcal B) =  {\rm T}({\rm T}^\ast \mathcal B) \oplus {\rm T}({\rm T}^\ast \mathcal B).$ 
Now $P$ is a linear map in the description (\ref{mapP}), and this implies that $\mathrm dP$ maps a tangent vector $X$ at  $\mathcal{T} \in \mathcal A$ to a tangent vector $V$  at $P(\mathcal{T}) \in {\rm T}^\ast \mathcal B \times {\rm T}^\ast \mathcal B$ of the form
$$V = (X_0 +X_2, X_1+X_3, X_0-X_2,X_1-X_3).$$
Plugging $V,W$ as above, for some $X,Y$ tangent to $\mathcal{T}\in \mathcal A$, into the formula for $\Omega$, gives the result claimed, i.e. 
$$\Omega(V,W) = \omega_I(X,Y).$$
\end{proof}

In the coordinates on ${\rm T}^\ast \mathcal B \times {\rm T}^\ast \mathcal  B$ provided by the description (\ref{cotanB}), the Nahm--Schmid equations are then equivalent to the system
\begin{eqnarray}
\dot B_1 + \left[A_1, B_1\right] &=& 0 \nonumber \\
\dot B_2 + \left[ A_2, B_2\right] &=& 0 \label{paracxreal} \\
\dot A_1 - \dot A_2 + \left[\frac{1}{2}(A_1+A_2),A_1- A_2\right] - \left[B_1,B_2\right] &=& 0, \nonumber
\end{eqnarray}
when we set
$$A_1 := T_0+T_2,\quad A_2 := T_0-T_2,\quad B_1 := T_1+T_3,\quad B_2 := T_1-T_3.$$
We may view the first two equations in (\ref{paracxreal}) as a single $\gf\oplus\gf$-valued ``paracomplex" equation, and the third equation as a ``real" equation. It is evident that the paracomplex equation is invariant under what one may call {\em paracomplex gauge transformations}, i.e. elements of $\mathcal G\times \mathcal G$ acting componentwise as
$$(u_1,u_2).\left(\frac{\mathrm d}{\mathrm dt} + A_i, B_i\right)_{i=1}^2 = \left(\frac{\mathrm d}{\mathrm dt} + u_i^{-1}A_iu_i +u_i^{-1}\dot u_i,u_i^{-1} B_iu_i\right)_{i=1}^2.$$
Let us consider the moduli space $\mathcal P$ given by
$$\mathcal P := \{(A,B)\in T^\ast \mathcal B\ | \ \dot B+[A,B] = 0\}/\mathcal G_{00}.$$
We define, in analogy with Theorem~\ref{MGg3}, a bijection  
$$\mathcal P \to G \times \gf,\qquad (A,B) \mapsto (u(1),B(0)),$$
where again $u$ is the unique gauge transformation that satisfies $A = -\dot uu^{-1}$ and  $u(0) = 1_G$. We may view $\mathcal P$ as a symplectic quotient of ${\rm T}^\ast \mathcal B$ by the action of the group $\mathcal G_{00}$. This map then identifies $\mathcal P$ and ${\rm T}^\ast G$ as symplectic manifolds. By applying this construction on each factor of ${\rm T}^\ast \mathcal B \times {\rm T}^\ast \mathcal B$ and composing with the map $P$ from above, we obtain a map from the moduli space of solutions to the Nahm--Schmid equations to $\mathcal P \times \mathcal P\cong {\rm T}^\ast G\times {\rm T}^\ast G$.

The paracomplex equation in the system (\ref{paracxreal}) may be interpreted as saying that, for $i=1,2$, the connections $$\nabla_i := \left(\frac{\mathrm d}{\mathrm dt} + A_i\right)\mathrm dt + B_i \, \mathrm dx$$ on the adjoint bundle over Euclidean $\R^2$ are \emph{flat}. 

We note that any solution to the paracomplex equation with $A_1 = A_2$ and $B_1 = B_2$ automatically solves the real equation. The corresponding solutions to the Nahm--Schmid equations satisfy $T_2 = 0 = T_3$. We can moreover rewrite the real equation as 
$$\dot A_1 - \dot A_2 + \left[A_2,A_1- A_2\right] + \left[B_2,B_1-B_2\right] = 0,$$
which says that the two flat connections $\nabla_1,\nabla_2$ are in \emph{Coulomb gauge} with respect to each other, i.e. 
$$\mathrm d^{\nabla_2} *(\nabla_1-\nabla_2) = 0.$$ 

This gives a natural interpretation of the Nahm--Schmid equations in terms of the geometry of the infinite-dimensional Riemannian manifold ${\rm T}^\ast \mathcal B/\mathcal G_{00}$, with Riemannian metric given by the usual $L^2$-metric. 

The  $L^2$-metric equips the affine space ${\rm T}^\ast \mathcal B$ with the structure of a flat infinite-dimensional Riemannian manifold, so geodesic segments $\gamma: [0,1]\to {\rm T}^\ast \mathcal B$ starting at $(A_1,B_1)$ with initial velocity $(a_1,b_1)$ are just straight lines: 
$$\gamma(\tau) = (A_1,B_1) + \tau(a_1,b_1), \quad \tau\in [0,1].$$
The group $\mathcal G_{00}$ acts by isometries and geodesic segments on the quotient ${\rm T}^\ast \mathcal B/\mathcal G_{00}$ can be lifted (at least locally) to horizontal geodesics on ${\rm T}^\ast \mathcal B$, i.e. geodesics, the velocity vector of which  is orthogonal to the $\mathcal G_{00}$ orbits. A geodesic $\gamma$ as above will be horizontal exactly when $(a_1,b_1)$ satisfies the Coulomb gauge condition. This means that a solution to the Nahm--Schmid equations on $[0,1]$ can be interpreted as a horizontal geodesic segment in ${\rm T}^\ast \mathcal B$, the end points of which lie in the space of solutions to the equation 
$$\dot B + [A,B] = 0.$$
Now a calculation similar to the one at the end of the previous section shows that if $\nabla_1,\nabla_2$ defines a solution to the Nahm--Schmid equations then the linearisation of the equation 
$$\mathrm d^{\nabla_1}*({\rm e}^\xi.(\nabla_2)-\nabla_1)=0$$ is again 
$$\Delta_\mathcal T\xi =0.$$
Thus, elements of the degeneracy locus correspond to horizontal geodesic segments, the endpoints of which are conjugate. The degeneracy locus $\mathcal D$ therefore has a natural interpretation in terms of the cut-locus of the infinite-dimensional Riemannian manifold ${\rm T}^\ast \mathcal B/\mathcal G_{00}$. 

Given a pair of points in $\mathcal P$ that are sufficiently close an application of the implicit function theorem shows that there exists a unique gauge transformation close to the identity that puts them into Coulomb gauge.   It follows that we can identify an open subset of $\mathcal M^0=(\mu^{-1}(0)\setminus\mathcal D)/\mathcal G_{00}$ with a  neighbourhood of the diagonal inside ${\rm T}^\ast G \times {\rm T}^\ast G$. In summary, we have the following observation:

\begin{proposition}\label{Prod2}
Consider the symplectic manifold $(\mathcal M^0,\omega_I)$ with local product structure $S$. 
The map $P$ defined in Proposition \ref{prop:Prod} induces a symplectomorphism between a neighbourhood of the subset $\{\mathcal T\in\mathcal M\ | \ T_2=T_3\equiv 0\}$ and a neighbourhood of the diagonal in $({\rm T}^\ast G\times {\rm T}^\ast G, \omega_{{\rm T}^\ast G}\oplus \omega_{{\rm T}^\ast G})$. This symplectomorphism is compatible with the local product structure $S$ on $\mathcal M^0$ and the obvious product structure on ${\rm T}^\ast G\times {\rm T}^\ast G$.
\end{proposition}

\section{Twistor space and spectral curves}\label{Lax}

This section describes how general constructions and results on paraconformal manifolds and on integrable systems with spectral parameter apply to the moduli space of solutions of the Nahm--Schmid equations.

\subsection{Twistor space}
Bailey and Eastwood \cite{BE} gave a very general construction of twistor space in the holomorphic category and the usual twistor space constructions for real manifolds can be viewed 
as applying a particular real structure to a Bailey--Eastwood twistor space. In the case of hypersymplectic manifolds this works as follows.
\par
The complexification of the algebra of split quaternions is the same as the complexification of the usual quaternions. Thus on the complexified tangent bundle ${\rm T}^\cx M$ of a hypersymplectic manifold $(M,g,I,S,T)$ we obtain three anti-commuting endomorphisms $I,J={\rm i}S,K={\rm i}T$ which all square to $-1$. For any point $p=(a,b,c)$ on the sphere $S^2\subset \R^3$, we consider the $(-{\rm i})$-eigenspace $E_p$ of the endomorphism $aI+bJ+cK$. This is an integrable distribution of ${\rm T}^\cx M$, since $I,S,T$ are integrable.  
Since a hypersymplectic manifold is real-analytic, it has a complexification $M^\cx$. The holomorphic tangent bundle of $M^\cx$ restricted to $M$ is naturally identified with ${\rm T}^\cx M$, and we can extend the endomorphisms $I,J,K$ uniquely to a neighbourhood of $M$ in $M^\cx$. For each $p\in S^2$ we obtain an integrable holomorphic distribution $E_p$ in ${\rm T}^{1,0} M^\cx$, and these combine to give an an integrable holomorphic distribution $E$ on $M^\cx\times \mathbb{CP}^1$. 
The twistor space $Z$ of $M$, as defined by Bailey and Eastwood, is the space of leaves of all $E$.  Let us discuss particular features of this construction in the case of hypersymplectic manifolds.
\par
First of all, we observe that, if $a\neq 0$, then the equation $ (aI+b{\rm i}S+c{\rm i}T)u=-{\rm i}u$ can be rewritten as $(a^{-1}I+ca^{-1}S-ba^{-1}T)u=-{\rm i}u$. The endomorphism $a^{-1}I+ca^{-1}S-ba^{-1}T$ is, as can be readily checked, a complex structure on $M$.

Thus for $a\neq 0$, the space of leaves of $E_p$ is just $M$ equipped with the corresponding complex structure $a^{-1}I+ca^{-1}S-ba^{-1}T$. If $a=0$, $E_p$ is the complexification of the $(-1)$-eigenspace (in ${\rm T}M$) of the product structure $bS +cT$ on $M$, and its space of leaves  does not have to be Hausdorff. If, however, the space of leaves of the $(-1)$-eigenspace of $bS +cT$ is a manifold $M^+_{b,c}$, then the space of leaves of $E_{(0,b,c)}$ on $M^\cx$ can be identified (possibly after making $M^\cx$ smaller) with the complexification of $M^+_{b,c}$. Thus in this case 
$Z$ is a manifold. This is the case on a neighbourhood of the subset $\{\mathcal T\in\mathcal M\ | \ T_2=T_3\equiv 0\}$, as shown in Proposition \ref{Prod2}. We shall soon see that the moduli space $\mathcal M$ of Nahm--Schmid equations has a globally defined twistor space.

\par
The space $Z$ is clearly a complex manifold, and the projection onto $S^2\simeq \mathbb{CP}^1$ is holomorphic. It has a natural anti-holomorphic involution $\sigma$ obtained from the real structure on $M^\cx$ and the inversion $\tau$ with respect to the circle $a=0$ (i.e. $|\zeta|=1$) on $\mathbb{CP}^1$, $$\sigma(m,\zeta) = \left(\bar m,\frac{1}{\bar\zeta}\right).$$
\par
The manifold $M$ can be now recovered as a connected component of the Kodaira moduli space of real sections (i.e. equivariant with respect to $\tau$ on $\mathbb{CP}^1$ and $\sigma$ on $Z$) the normal bundle of which splits as a sum of $\mathcal{O}(1)$'s. The endomorphisms $I,S,T$ are also easily recovered, since the fibre of $Z$ over any $(a,b,c)$ with $a\neq 0$ is canonically biholomorphic to $(M, a^{-1}I+ca^{-1}S-ba^{-1}T)$, and any real section in our connected component of the Kodaira moduli space meets such a fibre in exactly one point. Observe that, over the circle $a=0$, it is no longer true that  real sections separate points of the fibres. All we can say that such sections locally separate points in $$\bigl(Z_{(0,b,c)}\bigr)^\sigma\times \bigl(Z_{(0,-b,-c)}\bigr)^\sigma\simeq M^+_{b,c}\times M^+_{-b,-c},$$ which is equivalent to $(M, bS+cT)$ being locally isomorphic to $M^+_{b,c}\times M^+_{-b,-c}$.
\begin{remark} Unlike in the case of hyperk\"ahler manifolds, the twistor space is not uniquely determined by $M$, since over the circle $|\zeta|=1$  only the $\tau$-invariant part of the fibres of $Z \rightarrow \C\PP^1$ is  uniquely determined. Even more generally, we shall refer to any complex manifold $Z$ with a holomorphic submersion onto $\C\PP^1$ and an anti-holomorphic involution $\sigma$ covering the inversion $\tau$ as a {\em twistor space}. It should be pointed out, however, that if we start with such a generalised $Z$, and apply the construction given at the beginning of this section to  a connected component $M$ of the Kodaira moduli space of real sections of $Z$, we shall in general not recover $Z$. All we can say is that the resulting twistor space $Z_M$ of $M$ is equipped with a local biholomorphism $\phi:Z_M\to Z$, which  is a fibre map (over $\cx\oP^1$) and which intertwines the real structures.\label{open}
\end{remark}
To recover the metric from $Z$, we need one more piece of data: a twisted holomorphic symplectic form along the fibres of $Z\to \mathbb{CP}^1$.  If we extend the metric $g$ to the complexified tangent bundle ${\rm T}^\cx M$ and define $\omega_I=g(I\cdot,\cdot)$, $\omega_J=g(J\cdot,\cdot)$, $\omega_K=g(K\cdot,\cdot)$, then it follows easily that the $\mathcal O(2)$-valued $2$-form
$$ \omega_+=\omega_J + {\rm i}\omega_K +2\zeta \omega_I - \zeta^2(\omega_J -{\rm i}\omega_K)$$
vanishes on vectors in $E_\zeta$. Here $\zeta$ is the standard complex affine coordinate on $\mathbb{CP}^1$ related to $p=(a,b,c)$ via
$$(a,b,c) = \left(\frac{1 -\zeta\bar \zeta}{1+\zeta \bar \zeta}, \frac{-(\zeta + \bar \zeta)}{1+\zeta \bar \zeta},\frac{{\rm i}(\bar\zeta-\zeta)}{1+\zeta \bar \zeta}\right) .$$
It follows that $\omega_+$ descends to fibres of the projection $Z\to \mathbb{CP}^1$. Since $J={\rm i}S$ and $K={\rm i}T$, we can rewrite $\omega_+$ in terms of the symplectic structures $\omega_I,\omega_S,\omega_T$ on $M$:
$$\omega_+={\rm i}\left(\omega_S + {\rm i}\omega_T -2{\rm i}\zeta \omega_I - \zeta^2(\omega_S -{\rm i}\omega_T)\right).$$
We can of course replace $\omega_+$ by $-{\rm i}\omega_+$ and hence, from now on, we shall consider the following holomorphic symplectic form along the fibres of $Z$:
\begin{equation}\label{Omega} \omega_+=\omega_S + {\rm i}\omega_T - 2{\rm i}\zeta \omega_I - \zeta^2(\omega_S -{\rm i}\omega_T).\end{equation}

\begin{remark} Away from the locus $|\zeta|\neq 1$, the twistor space can be also constructed in analogy with Hitchin's original construction of the twistor space of a hyperk\"ahler manifold in \cite{HKLR:1987}. First, we identify $\mathbb{CP}^1\backslash \{|\zeta|=1\}$ with the two-sheeted hyperboloid
$$ C=\{(x_1,x_2,x_3)\in \oR^3\ | \ x_1^2 -x_2^2-x_3^2 = 1\}$$ via 
$$x=(x_1,x_2,x_3) = \left(\frac{1 +\zeta\bar \zeta}{1-\zeta \bar \zeta}, \frac{{\rm i}(\bar\zeta-\zeta)}{1-\zeta\bar\zeta},\frac{\zeta + \bar \zeta}{1-\zeta \bar \zeta}\right).$$
Then we define an almost complex structure on the product $Z^{ o}:=M\times C$, which at $(m,x)$ is equal to $(x_1I+x_2S+x_3 T)\oplus J_0$, where $J_0$ is the standard complex structure on $\mathbb{CP}^1$. An argument completely analogous to the one in \cite{HKLR:1987} shows that this almost complex structure is integrable, and $Z^{ o}$ becomes a complex manifold fibring over $C$. Unlike in the case of hyperk\"ahler geometry, there is however no easy way to recover $M$ from $Z^{ o}$.
\end{remark}

Assume now that a Lie group $G$ acts on $M$ preserving the hypersymplectic structure, and that this action extends to a global action of a complexification $G^\cx$ of $G$. If the $G$-action is Hamiltonian for $\omega_I,\omega_S,\omega_T$, then the action of $G^\cx$ is Hamiltonian for the symplectic form \eqref{Omega},  and we obtain an associated moment map
$$\mu_+ = \mu_S + {\rm i}\mu_T -2{\rm i}\zeta \mu_I - \zeta^2(\mu_S -{\rm i}\mu_T).$$
 The fibrewise complex-symplectic quotient $\bar{Z}$ of $Z$ by $G^\cx$ (assuming it is manifold) is then a twistor space in the sense of Remark \ref{open}. The space of real sections of $\bar{Z}$ which descend from $Z$ is the hypersymplectic quotient of $M$ by $G$ (note however that the hypersymplectic quotient does not need to be a hypersymplectic manifold everywhere: some of these sections may have a wrong normal bundle).
 
We shall now apply the ideas above to the infinite dimensional context of the Nahm--Schmid equations. The twistor space of the flat infinite-dimensional hypersymplectic Banach manifold $\mathcal A$ is formally the total space of $C^1(U,\gf^\cx)\otimes \cx^2\otimes \mathcal O(1)$.
The fibrewise moment map for the complexified gauge group $\mathcal G_{00}$ with respect to $\omega_+$ is the Lax equation discussed in \S\ref{complex}. Thus each fibre of the quotient twistor space is isomorphic to $\gf^\cx\times G^\cx$. A computation completely analogous to that in \cite{Kronheimer:1988} shows that the twistor space $Z$ of $\mathcal M$ is obtained by gluing two copies of $\cx\times \gf^\cx\times G^\cx$ by means of the transition function
$$ (\zeta,\eta, g) \mapsto \left( \frac{1}{\zeta},\frac{\eta}{\zeta^2},g\exp(-\eta/\zeta) \right).
$$
The real structure is given  by
$$\sigma(\zeta,\eta,g)=\bigl(1/\bar\zeta,\bar\eta/\bar\zeta^2,(g^{\ast})^{-1}\exp(-\bar\eta/\bar\zeta)\bigr),
$$
and the complex symplectic form along the fibres is the standard symplectic form on ${\rm T}^\ast G^\cx$.

Real sections of $Z$ are obtained from real sections of $C^1(U,\gf^\cx)\otimes \cx^2\otimes \mathcal O(1)$ via the fibrewise quotient with respect to $\omega_+$.
If we write 
$$\alpha =  T_0-{\rm i}T_1,\qquad \beta  = T_2+{\rm i}T_3,$$  
then the moment maps of Proposition \ref{mmaps} can be written as:
\begin{eqnarray*}
2{\rm i}\mu_I(\mathcal T) &=& [\frac{\mathrm d}{\mathrm dt} + \alpha, -\frac{\mathrm d}{\mathrm dt}+\alpha^\ast ] -[\beta,\beta^\ast ] =  \dot\alpha + \dot\alpha^\ast  + [\alpha,\alpha^\ast ] - [\beta,\beta^\ast ] \\
(\mu_S +{\rm i}\mu_T)(\mathcal T)&=& [\frac{\mathrm d}{\mathrm dt} + \alpha, \beta] =  \dot\beta + [\alpha,\beta] \\
(\mu_S -{\rm i}\mu_T)(\mathcal T) &=& [-\frac{\mathrm d}{\mathrm dt} + \alpha^\ast , \beta^\ast ] = -\dot\beta^\ast  + [\alpha^\ast ,\beta^\ast ].
\end{eqnarray*}
Putting the terms together, we obtain for $\mu_+$
$$\mu_+(\mathcal T) = \mu_S +{\rm i}\mu_T -2{\rm i}\zeta \mu_I -\zeta^2(\mu_S -{\rm i}\mu_T) = \left[\frac{\mathrm d}{\mathrm dt} + \alpha -\zeta \beta^\ast , \beta -\zeta \left(-\frac{\mathrm d}{\mathrm dt} +\alpha^\ast \right)\right].$$
Equivalently, we can modify this to
$$\mu_+(\mathcal T) =\left[\frac{\mathrm d}{\mathrm dt} + \alpha -\zeta \beta^\ast , \beta -\zeta \left(\alpha +\alpha^\ast \right)+\zeta^2\beta^\ast  \right].$$

Thus the real sections of $Z$ obtained in this way are precisely gauge equivalence classes of solutions to the Nahm--Schmid equations, i.e.\ points of $\mathcal M$. However, only the sections corresponding to $\mathcal M\backslash \mathcal D$ have normal bundle with the required splitting property.

\subsection{The spectral curve\label{spectral}}

The expression for $\mu_+(\mathcal T)$ obtained above means that the Nahm--Schmid equations are equivalent to the following single ODE with a complex parameter:
\begin{equation}\label{Lax-spectral} \frac{\rm d}{{\rm d}t}\left(\beta -\zeta (\alpha +\alpha^\ast)+\zeta^2\beta^\ast\right)=\left[    \beta -\zeta \left(\alpha +\alpha^\ast \right)+\zeta^2\beta^\ast,\ \alpha -\zeta \beta^\ast  \right].\end{equation}
This is an example of a Lax equation with spectral parameter: if we write 
$$T(\zeta)=\beta -\zeta (\alpha +\alpha^\ast)+\zeta^2\beta^\ast, \ T_+(\zeta)=\alpha -\zeta \beta^\ast,$$
then \eqref{Lax-spectral} becomes
\begin{equation}\label{Laxs} \frac{{\rm d}T(\zeta)}{{\rm d}t}=\left[T(\zeta),T_+(\zeta)\right].\end{equation}
This equation has been studied extensively (see e.g. \cite{AHH}), and so we can simply state the relevant results. Here we shall restrict ourselves to the case $\gf=\un(n)$.
\par
First of all, \eqref{Laxs} implies that the spectrum of $T(\zeta)$ is constant in $t$, so that the characteristic polynomial of $T(\zeta)$ defines a conserved quantity of the Nahm--Schmid system:
$$  \Sigma=\left\{(\zeta,\eta);\ \det(\eta-T(\zeta))=0\right\}.$$
This is an algebraic curve (possibly singular) of (arithmetic) genus $(n-1)^2$ in the total space of the line bundle $\mathcal O(2)$ over $\mathbb{CP}^1$, invariant under the involution 
\begin{equation}\label{tau}\sigma(\zeta,\eta)=\bigl(1/\bar\zeta,\bar\eta/\bar\zeta^2\bigr).
\end{equation}
In particular, we observe that the coefficient of $\zeta^2$ in $\tr T(\zeta)^2$ is the conserved quantity $C = 2\|T_1\|^2 + \|T_2\|^2 + \|T_3\|^2$ introduced in Corollary \ref{GlobalReg}.
\par
 The cokernel of $\eta-T(\zeta)$ is a semistable sheaf $F$ on $\Sigma$ with $\chi(F)=n^2-n$ (and such that $H^0(\Sigma,F\otimes \mathcal O(-1))=0$). The Lax equation corresponds then to the linear flow $t\to F\otimes L^t$, where
 $L$ is a line bundle on $T\mathbb{CP}^1\simeq |\mathcal O(2)|$ with the transition function $\exp(\mu(\zeta,\eta))$ and $\mu$ is polynomial in $\zeta^{-1}$ and $\eta$ determined by the condition that the positive powers of $\zeta$ in $\mu\bigl(\zeta,T(\zeta)\bigr)$ and in $-T_+(\zeta)$ coincide \cite{AHH}. In our case, it follows that $\mu(\zeta,\eta)=\eta/\zeta$.
 \par
 The sheaves arising from solutions to the Nahm--Schmid equations satisfy several further conditions. 
 Consider the case of a smooth spectral curve $\Sigma$. Then the cokernel of $\eta-T(\zeta)$ is $1$-dimensional for each $(\zeta,\eta)\in \Sigma$ and, consequently, $F$ is a line bundle. Moreover $F(-1)= F\otimes \mathcal O(-1)$ is a line bundle of degree $g-1$ not in the theta divisor. The fact that $F$ is obtained from the matrix polynomial $T(\zeta)$ imposes two further conditions:
 \begin{itemize}
 \item[(i)] $E=F(-1)$ satisfies the reality condition $E\otimes \tau(E)\simeq K_ S\simeq \mathcal O(2n-4)$, where $\tau$ is the induced antiholomorphic involution on $\Jac^{g-1}(\Sigma)$;
 \item[(ii)] the Hitchin metric \cite[\S 6]{Hitchin:1983} on $H^0(\Sigma,F)$ is positive definite.
 \end{itemize}
 It is perhaps worth pointing out that the fact that the solutions to the Nahm--Schmid equations exist for all time implies that the closure of the set of line bundles satisfying the above conditions does not intersect the theta divisor in $\Jac^{g-1}(\Sigma)$.

 \begin{example}[Spectral curves for $\gf=\mathfrak{su}(2)$]
We calculate the spectral curve associated with a solution  $(T_0,T_1,T_2,T_3) = (0, f_1e_1, f_2e_2, f_3e_3)$ as in \eqref{PauliAnsatz}. With
\begin{equation}\label{basis} e_1=\frac{1}{2}\begin{pmatrix} {\rm i} & 0\\0 & -{\rm i}\end{pmatrix},\quad e_2=\frac{1}{2}\begin{pmatrix} 0 & 1\\-1 & 0\end{pmatrix}, \quad e_3=\frac{1}{2}\begin{pmatrix} 0 & {\rm i}\\{\rm i} & 0\end{pmatrix},
\end{equation}
 we get 
$$T(\zeta)=\frac{1}{2}\begin{pmatrix} \ -2f_1\zeta & f_2(1-\zeta^2)-f_3(1+\zeta^2)\\ \\f_2(\zeta^2-1)-f_3(1+\zeta^2) & 2f_1\zeta\ \end{pmatrix}.$$
A direct computation yields
$$\det (\eta - T) = \eta^2 - \frac{1}{2}(2f_1^1+f_2^2+f_3^2)\zeta^2 +\frac{1}{4}(f_2^2-f_3^2)(1+\zeta^4).$$
Taking $f_1(t) = a \kappa \, \mathrm{sn}_\kappa(at+b), f_2(t) = a \kappa \, \mathrm{cn}_\kappa(at+b), f_3(t) = - a\, \mathrm{dn}_\kappa(at+b)$ we can write this as 
$$\det (\eta - T) =  \eta^2 -\frac{a^2}{2}(1+\kappa^2)\zeta^2 + \frac{a^2}{4}(\kappa^2-1)(1+\zeta^4).$$
In this case, $g=1$ and we can identify $\Jac^0(\Sigma)$ with $\Sigma$ itself. If we choose the origin (trivial bundle) $O$ to correspond to a $\tau$-invariant point on $\Sigma$, then the induced anti-holomorphic involution on $\Jac^0(\Sigma)\simeq  \Sigma$ is just $\tau$. According to the above characterisation of the sheaves $F$, they correspond to points $x\in \Sigma$ such that $x+\tau(x)=O$. This is a union of two circles in $\Sigma$, and the Nahm--Schmid flow lives on the circle that does {\em not} contain $O$. We also observe that the points where $f_1$ or $f_2$ vanish (i.e.\ points for which $t$ lies in $-b+({\rm K}(\kappa)/a)\oZ$) correspond to the even theta characteristics of $\Sigma$.
\end{example}

\section{The ``positive" part of the moduli space}
 
We continue to assume that $\gf=\mathfrak{u}(n)$. To any triple $T_1,T_2,T_3\in \mathfrak{u}(n)$ we associate, as in the previous section,
 the matrix polynomial $T(\zeta)=T_2+{\rm i}T_3+2{\rm i}T_1\zeta+(-T_2+{\rm i}T_3)\zeta^2$. It follows that $T(\zeta)\zeta^{-1}$ is hermitian for each $\zeta$ with $|\zeta|=1$ and we denote by $X^+$ the set of matrix polynomials $T(\zeta)$ such that $T(\zeta)\zeta^{-1}$ is positive definite for each $\zeta$ with $|\zeta|=1$. Observe that, for $n=1$, this is the subset of $\R^{1,2}$ where $x^2-z\bar z>0$, $x>0$.
 \par
 It follows from a theorem of Rosenblatt \cite{Ros} (see also \cite{Mal} for a simple proof) that any $T(\zeta)\in X^+$ can be factorised  as 
 \begin{equation}T(\zeta)=(A+B^\ast\zeta)(B+A^\ast \zeta),\label{AB}\end{equation}
 where $\det(B+A^\ast\zeta)\neq 0$ for $|\zeta|\leq 1$. The only freedom in the factorisation is the following  action of $\mathrm{U}(n)$: $(A,B)\mapsto (Ag^{-1}, gB)$ and, consequently, the factorisation is unique if we assume, in addition, that $B$ is hermitian.
 The following estimate is one reason for our interest in $X^+$.
 \begin{lemma} If $(T_1,T_2,T_3)\in X^+$, then $\|T_2+{\rm i}T_3\|\leq 2\|T_1\|$. \label{bound}
\end{lemma}
\begin{proof} We have $$ \|T_2+{\rm i}T_3\|=\|AB\|\leq \|A\|\|B\|\leq \|2{\rm i}T_1\|^{1/2} \|2{\rm i}T_1\|^{1/2}=2\|T_1\|.$$
\end{proof}
 
 We now consider the Nahm--Schmid flow on $X^+$.
 \begin{definition} The {\em positive} part $\sM^+$ of the moduli space $\sM$ consists of solutions $(T_0(t),T_1(t),T_2(t),T_3(t))$ to the Nahm--Schmid equations such that $(T_1(t),T_2(t),T_3(t))\in X^+$ for some $t$ (equivalently for all $t$).
 \end{definition}

Suppose that $T_1(t),T_2(t),T_3(t)$ satisfy the Nahm--Schmid equations (with $T_0\equiv 0$) and that we have a factorisation \eqref{AB}  depending smoothly on $t$ (we can always find a smooth factorisation, given its uniqueness under the assumptions $B=B^\ast$, $\det(B+A^\ast\zeta)\neq 0$ for $|\zeta|\leq 1$). We can then rewrite equation \eqref{cx} as 
 $$ \left(\dot A-\frac{1}{2}B^\ast B A+\frac{1}{2}ABB^\ast\right)B+A\left(\dot B-\frac{1}{2}A^\ast AB+\frac{1}{2}BAA^\ast\right)=0,
 $$
 and equation \eqref{re} as
 $$ \left(\left(\dot A-\frac{1}{2}B^\ast B A+\frac{1}{2}ABB^\ast\right)A^\ast+B^\ast \left(\dot B-\frac{1}{2}A^\ast AB+\frac{1}{2}BAA^\ast\right)\right)^H=0, $$
 where the superscript $H$ denotes the hermitian part of a matrix. Thus $A(t),B(t)$ satisfy the following equations:
 \begin{eqnarray}\label{AA} \dot A & = & \frac{1}{2}\left(B^\ast B A-ABB^\ast\right)+a,\\ \label{BB}
 \dot B& = &\frac{1}{2}\left(A^\ast AB-BAA^\ast\right)+b,
 \end{eqnarray}
 where $a(t),b(t)$ satisfy
 \begin{equation}\label{ker} a B+Ab=0,\quad a A^\ast+Aa^\ast+b^\ast B+B^\ast b=0\end{equation}
 for all $t$. 

 \begin{lemma} Let $(a,b)$ be a solution of \eqref{ker} and suppose that $\det(B+A^\ast\zeta)\neq 0$ for $|\zeta|\leq 1$. 
 Then there exists a unique $\rho\in \mathfrak{u}(n)$ such that  $a=A\rho$, $b=-\rho B$. 
 \end{lemma}
 {\flushleft{\em Proof.}} The assumption implies that $B$ is invertible. We can infer from the first equation in \eqref{ker} that there exists a unique $\rho\in \mathfrak{gl}(n,\cx)$ such that $a=A\rho$, $b=-\rho B$. The second equation can be now rewritten as $$(B^\ast)^{-1}A(\rho+\rho^\ast)A^\ast B^{-1}=\rho+\rho^\ast .$$ The conclusion follows from the following lemma:
\begin{lemma} Let $X$ be a square complex matrix, any two eigenvalues of which satisfy $\lambda_1\ol\lambda_2\neq 1$. Then the only solution of the equation $X^\ast h X=h$ with $h$ hermitian is $h=0$.
\end{lemma}
\begin{proof} We may assume that $X$ is in Jordan form. The equation under consideration means that $X$ is an isometry of the sesquilinear form $\langle\,,\,\rangle$ 
defined by $h$. Let $J_k(\lambda_1), J_l(\lambda_2)$ be two (not necessarily distinct) Jordan blocks of $X$ with $e_1,\dots,e_k$ and $f_1,\dots,f_l$ the corresponding bases of cyclic $X$-modules, i.e. $Xe_1=\lambda_1 e_1$, $Xe_i=\lambda_1 e_i+e_{i-1}$ for $1<i\leq k$, and similarly for the $f_j$. Since $\ol\lambda_1\lambda_2\neq 1$, $\langle e_1,f_1\rangle=0$. Now it follows inductively that $\langle e_i,f_j\rangle=0$ for all $i\leq k,j\leq l$. Thus $h=0$. 
\end{proof}

 We can therefore rewrite \eqref{AA} and \eqref{BB} as
 \begin{eqnarray*} \dot A -A\rho & = & \frac{1}{2}\left(B^\ast B A-ABB^\ast\right),\\ 
 \dot B+\rho B& = &\frac{1}{2}\left(A^\ast AB-BAA^\ast\right).
 \end{eqnarray*}
 These equations are invariant under the following action of $\mathrm U(n)$-valued gauge transformations:
 $$ A\mapsto Ag^{-1},\enskip B\mapsto g B,\enskip \rho\mapsto g\rho g^{-1}-\dot gg^{-1}.$$
 In particular, we can use the gauge freedom to make $\rho$ identically $0$ and obtain the following equations for $A,B$:
 \begin{eqnarray}\label{AAA} \dot A & = & \frac{1}{2}\left(B^\ast B A-ABB^\ast\right),\\ \label{BBB}
 \dot B& = &\frac{1}{2}\left(A^\ast AB-BAA^\ast\right).
 \end{eqnarray}
 These equations are the split signature version of the Basu--Harvey--Terashima equations \cite{Ter,Bpre}. We conclude:
 \begin{theorem} Let $(T_1(t),T_2(t),T_3(t))$, $t\in \R$, be a solution to the reduced Nahm--Schmid equations \eqref{NahmSchmidred} belonging to $\sM^+$. Then there exist smooth functions $A,B:\R\to \Mat_{n,n}(\cx)$ satisfying \eqref{AAA} and \eqref{BBB} such that
 \begin{equation}\label{TAB} T_1(t)=-\frac{{\rm i}}{2}\left(A(t)A^\ast(t)+B^\ast(t)B(t)\right), \quad T_2(t)+{\rm i}T_3(t)=A(t)B(t).\end{equation}
 Conversely, if $A(t),B(t)$ satisfy \eqref{AAA}-\eqref{BBB}, then $T_1(t),T_2(t),T_3(t)$ given by \eqref{TAB} satisfy \eqref{NahmSchmidred}.
\hfill $\Box$
 \end{theorem} 

 \begin{remark} 1. The quantity $\tr(AA^\ast+B^\ast B)$ is an invariant of the flow \eqref{AAA}-\eqref{BBB}. It is therefore a flow on a sphere.\newline
 2. The flow \eqref{AAA}-\eqref{BBB} is preserved under the involution $A\leftrightarrow B$. In particular, there is a subset of solutions with $A=B$, i.e.
 satisfying 
 $$\dot A= \frac{1}{2}\left(A^\ast A^2-A^2A^\ast\right)=\frac{1}{2}\left[A^\ast A+AA^\ast, A\right].$$
 \end{remark}
 
 As an application, we obtain additional information about the product structures of $\sM^+$.
 \begin{proposition} The map $\bigl(\ol{\sM^+},S\bigr)\to {\rm T}^\ast G\times {\rm T}^\ast G$ defined in \S\ref{product} is proper.
 \end{proposition}
 \begin{proof} The map under consideration is given by
 $$\phi(T_0(t),T_1(t),T_2(t),T_3(t))= \bigl(T_1(0)+T_3(0),u_1(1),T_1(0)-T_3(0),u_2(1)\bigr),$$
 where $u_1(t)$ (resp. $u_2(t)$) satisfies $T_0+T_2=-\dot u_1u^{-1}_1$, $u_1(0)=1$ (resp. $T_0-T_2=-\dot u_2u^{-1}_2$, $u_2(0)=1$).
Let $K$ be a compact set in $T^\ast G\times T^\ast G$. Then $\|T_1(0)\|$ is bounded on $\phi^{-1}(K)\cap  \ol{\sM^+}$ by a constant $c_K$. Lemma \ref{bound} implies that $\|T_2(0)\|^2+\|T_3(0)\|^2 \leq 4c_K^2$, and hence the conserved quantity $C$ of Proposition~\ref{Conserved}  is bounded by $ 6c_K^2$. The Nahm--Schmid equations and the Arzel\`a-Ascoli theorem imply now that $\phi^{-1}(K)$ is compact in $\sM$.                                                                                                    
 \end{proof}
 
\section{Other moduli spaces of solutions to the Nahm--Schmid equations}

We want to briefly discuss two other moduli spaces: solutions  on a half-line and solutions on $[0,1]$ which are symmetric at both ends.

\subsection{The moduli space of solutions on a half-line}

The reduced Nahm--Schmid equations have the form $\dot{x}=V(x)$, where $x=(T_1,T_2,T_3)\in \gf^3$ and $V$ is the vector field given by $V(x)=\bigl([T_3,T_2], [T_3,T_1], [T_1,T_2]\bigr)$.  The critical points of $V$ consist of commuting triples $(T_1,T_2,T_3)$. Let $(\tau_1,\tau_2,\tau_3)$ be a {\em regular} commuting triple, i.e. the centraliser of the triple is a Cartan subalgebra $\hf$. Let $\gf=\bigoplus_{\lambda\in \hf^\ast} \gf_\lambda$ be the root decomposition of $\gf$ with respect to $\hf$. A simple calculation shows that the differential $DV$ at $(\tau_1,\tau_2,\tau_3)$ preserves the subspace $U_{\lambda}=\gf_\lambda\oplus \gf_\lambda \oplus \gf_\lambda$ and $DV_{|U_\lambda}$ has eigenvalues $0, \pm\bigl(\lambda(\tau_2)^2+\lambda(\tau_3)^2-\lambda(\tau_1)^2\bigr)^{1/2}$. Thus, as long as $(\ad \tau_2)^2+ (\ad \tau_3)^2-(\ad \tau_1)^2$ has no negative eigenvalues, the differential of $V$ at the critical point $(\tau_1,\tau_2,\tau_3)$ has no imaginary eigenvalues. We shall cal
 l such a triple (irrespectively of being regular or not) {\em stable}. Observe that, in particular, the triple $(\tau_1,0,0)$ is stable.
\par
General results about the asymptotic behaviour of the solutions of ODEs (see e.g.\ \cite{Kov} for a discussion of this in the context of Nahm's equations) imply that any solution to the Nahm--Schmid equations on $[0,\infty)$, the limit of which is a regular stable triple, approaches this limit exponentially fast, i.e.\ there exists an $\eta>0$ such that $\|T_i(t)-\tau_i\|\leq \text{ const}\cdot {\rm e}^{-\eta t}$. Any $\eta$ smaller than the smallest positive eigenvalue of $D_{(\tau_1,\tau_2,\tau_3)}V$ will do.
\par
Let $(\tau_1,\tau_2,\tau_3)$ be a regular stable triple and let $\eta>0$ be as above. We introduce (cf.\ \cite{KronJLMS}) the Banach space $\Omega_1$ consisting of $C^1$-maps $f:[0,\infty)\to \gf$ such that the norm
\begin{equation} \label{norm1} \|f\|_{\Omega_1}:=\sup_t \left\| {\rm e}^{\eta t}f(t)\right\|+ \sup_t \left\| {\rm e}^{\eta t}f^\prime(t)\right\|\end{equation}
is finite. Let us write $\tau_0=0$ and set 
$$\mathcal A_{\tau_1,\tau_2,\tau_3}:=\{(T_0,T_1,T_2,T_3):[0,\infty)\to \gf\otimes \oR^4\ |\ T_i-\tau_i\in \Omega_1,\enskip i=0,1,2,3\}.$$
The Banach group
$$\mathcal G=\{u:[0,\infty)\to G\ |\ g(0)=1,\;\dot gg^{-1}\in \Omega_1, \Ad(g) \tau_i\in \Omega_1,\enskip i=1,2,3\}$$
acts smoothly on the Banach manifold $\mathcal A_{\tau_1,\tau_2,\tau_3}$ via gauge transformations and we define
\begin{equation*} \mathcal M_{\tau_1,\tau_2,\tau_3}= \{\mathcal T\in\mathcal A\ | \ \text{$\mathcal T$ solves equations (\ref{NahmSchmidUNred})} \}/\mathcal G.
\end{equation*}
As in \cite{KronJLMS}, $\mathcal M_{\tau_1,\tau_2,\tau_3}$ can be also described as the set of solutions to the reduced Nahm--Schmid equations such that  $\lim_{t\to\infty} T_i(t)=\Ad(g_0)\tau_i$, $i=1,2,3$, for some $g_0\in G$.

\begin{theorem}
$\mathcal M_{\tau_1,\tau_2,\tau_3}$ is a smooth Banach manifold diffeomorphic to the tangent bundle ${\rm T}( G/T)$ of $G/T$, where $T$ is a maximal torus in $G$.  
\end{theorem}
\begin{proof} The fact that $\mathcal M_{\tau_1,\tau_2,\tau_3}$ is a smooth Banach manifold is proved analogously to Theorem \ref{MGg3}, once we replace the constant $X_0$ in the definition of the slice $\mathcal S_{\mathcal T}$ with $X_0(t)=v{\rm e}^{-2\eta t}$, $v\in \gf$. To identify $\mathcal M_{\tau_1,\tau_2,\tau_3}$ we observe that the space of solutions to the reduced Nahm--Schmid equations in a small neighbourhood of the critical point $(\tau_1,\tau_2,\tau_3)$ and converging to this critical point is identified with the sum of negative eigenspaces of $D_{(\tau_1,\tau_2,\tau_3)}V$ (since the triple is stable). By virtue of the arguments given at the beginning of this section, this space is isomorphic to $\p=\bigoplus_{\lambda>0}\gf_\lambda\oplus\gf_{-\lambda}$.
 Since every solution to the Nahm--Schmid equations exists for all time, we conclude that the submanifold of  $\mathcal M_{\tau_1,\tau_2,\tau_3}$ consisting of solutions with $T_0\equiv 0$ is diffeomorphic to $\p$. Now the remark made just before the statement of the theorem shows that $\mathcal M_{\tau_1,\tau_2,\tau_3}$ is diffeomorphic to $G\times_T \p\simeq {\rm T} (G/T)$.
\end{proof}

We now would like to view $\mathcal M_{\tau_1,\tau_2,\tau_3}$ as the hypersymplectic quotient of $\mathcal A_{\tau_1,\tau_2,\tau_3}$ by $\mathcal G$. As for the moduli space of solutions on an interval, the hypersymplectic structure is not defined on the degeneracy locus, i.e. on the set of solutions where the linear operator \eqref{linearop}, acting on ${\rm Lie}(\mathcal G)$, has nontrivial kernel.
We have:
\begin{proposition} Let $(\tau_1,\tau_2,\tau_3)$ be a regular stable triple. There exists a neighbourhood $U$ of $G/T$ (which is the $G$-orbit of the constant solution $(0,\tau_1,\tau_2,\tau_3)$) in $\mathcal M_{\tau_1,\tau_2,\tau_3}$ which is disjoint from the degeneracy locus $\mathcal D$.\end{proposition}
\begin{proof} It is enough to show that the constant solution itself does not belong to $\mathcal D$. As in the proof of Proposition 
\ref{Prop:DegenLoc} we are asking whether the exists a nonzero $\xi\in {\rm Lie}(\mathcal G)$ such that
$$\ddot\xi=\bigl((\ad \tau_2)^2+ (\ad \tau_3)^2-(\ad \tau_1)^2\bigr)(\xi).$$
We can diagonalise this equation and conclude, since the triple is stable, that $\|\xi\|^2$ is a convex function. Since $\xi(0)=0$ and $\xi$ has a finite limit at $t=\infty$, such a $\xi$ must be identically zero.
\end{proof}

As for the moduli space $\mathcal M$, we can conclude that, away from the degeneracy locus, $\mathcal M_{\tau_1,\tau_2,\tau_3}$ is a smooth hypersymplectic manifold.
Its complex and product structures are described by the following result, proved by combining the arguments of \S\ref{complex}--\ref{product} and  of \cite{KronJLMS}.
\begin{theorem} There exists a neighbourhood $U$ of $G/T$ in  $\mathcal M_{\tau_1,\tau_2,\tau_3}$ not intersecting the degeneracy locus $\mathcal D$ such that:
\begin{itemize}
\item[(i)] if $\tau_2+{\rm i}\tau_3$ is a regular semisimple element of $\gf^\cx$, then $(U,I)$ is biholomorphic  to an open neighbourhood of $\Ad (G)(\tau_2+{\rm i}\tau_3)$ in the complex adjoint orbit of $\tau_2+{\rm i}\tau_3$;
\item[(ii)] if $\tau_2=\tau_3=0$, then $(U,I)$ is biholomorphic   to an open neighbourhood of $G^\cx/B$ in ${\rm T}( G^\cx/B)$, where $B$ is a Borel subgroup of $G^\cx$;
\item[(iii)] with respect to the product structure $S$, $U$ is isomorphic to an open neighbourhood of the diagonal $G$-orbit of $(\tau_1+\tau_3,\tau_1-\tau_3)$ in $O_+\times O_-$, where $O_\pm$ is the adjoint $G$-orbit of $\tau_1\pm \tau_3$.
\end{itemize}
\end{theorem}

\begin{remark} 1. With extra care, one should be able to show, as in  \cite{Biq},  that in the case when $\tau_2+{\rm i}\tau_3$ is a nonzero nonregular element of $\gf^\cx$, $U$ is biholomorphic to a neighbourhood of the zero section in a vector bundle over $G^\cx/P$, where $P$ is an appropriate parabolic subgroup.\newline
2. Other complex and product structures of $U\subset \mathcal M_{\tau_1,\tau_2,\tau_3}$  are easily identified by observing that such a complex or product structure $W$ can be rotated via an element $A\in {\rm SO}(1,2)$ to $W_0=I$ or $W_0=S$, and that $(U,W)$ is isomorphic to $(U^\prime,W_0)$, where $U^\prime$ is the corresponding open subset of $\mathcal M_{A(\tau_1,\tau_2,\tau_3)}$. 
\end{remark}

\subsection{Hypersymplectic quotients by compact subgroups}

We return to $\mathcal M$, the moduli space of solutions on $[0,1]$. As observed in \S\ref{actions}, $\mathcal M$ possesses a hypersymplectic action of $G\times G$, given by allowing gauge transformations with arbitrary values at $t=0$ and $t=1$. The hypersymplectic moment map $(\mu_I,\mu_S,\mu_T)$ for this action is computed from the calculation in the proof of Proposition \ref{mmaps} and yields
$$ \mu_I(\mathcal T)=(-T_1(0),T_1(1)),\quad   \mu_S(\mathcal T)=(T_2(0),-T_2(1)),\quad      \mu_T(\mathcal T)=(T_3(0),-T_3(1)).
$$
We can now consider hypersymplectic quotients of $\mathcal M$ by any subgroup $K$ of $G\times G$, to be denoted by $\mathcal{N}$. If we decompose $\gf\oplus \gf$ orthogonally with respect to the $\Ad$-invariant metric as $\kf\oplus \mf$, where $\kf={\rm Lie}(K)$, then the hypersymplectic quotient is the moduli space of solutions to the Nahm--Schmid equations on $[0,1]$, such that $(T_i(0),T_i(1))\in \mf$, $i=1,2,3$, modulo gauge transformations such that $(g(0),g(1))\in K$. These moduli spaces can, of course, be singular, and even if they are smooth, their hypersymplectic structure may degenerate at some points.
\par
An interesting example is the case $G={\rm U}(n)$, $K={\rm O}(n)\times {\rm O}(n)$ (or $G={\rm SU}(n)$, $K={\rm SO}(n)\times {\rm SO}(n)$). The moduli space $\mathcal N$ consists of solutions such that  $T_i(0)$ and $T_i(1)$ are symmetric for $i\geq 1$, modulo orthogonal gauge transformations. This means that, at $t=0,1$, the quadratic matrix polynomial $T(\zeta)$ considered in \S\ref{spectral} is a symmetric matrix for each $\zeta$. 
It is not hard to see (cf. \cite[Prop. 1.4]{Biel}) that, under the correspondence between $T(\zeta)$ and semistable sheaves $F$ on the spectral curve considered in  \S\ref{spectral}, such symmetric matrix polynomials correspond to theta characteristics, i.e. the sheaf $E=F(-1)$ satisfies $E^2\simeq \omega_{\Sigma}$, where $ \omega_{\Sigma}$ is the dualising sheaf of $ S$ (here we assume that $F$ is an invertible sheaf on $\Sigma$, i.e. that the dimension of any eigenspace of $T(\zeta)$, $\zeta\in \mathbb{CP}^1$, is equal to $1$). Thus, solutions to the Nahm--Schmid equations in $\mathcal N$ correspond to flows between two theta characteristics in the direction $L^t$, where $L$ is the line bundle on ${\rm T}\mathbb{CP}^1$ with transition function $\exp(\eta/\zeta)$. 
\par
Denote by $F_t$ the sheaf corresponding to $T(\zeta)(t)$. Thus $F_1=L\otimes F_0$. Moreover both $F_0(-1)$ and $F_1(-1)$ are theta characteristics, i.e.\ they both square to $ \omega_{\Sigma}$. It follows that ${L^2}_{|\Sigma}\simeq \mathcal O$. This condition is well known to hold for spectral curves of ${\rm SU}(2)$-monopoles and it is therefore tempting to consider $\mathcal N$ as a hypersymplectic analogue of the hyperk\"ahler monopole moduli space. To see  that such an interpretation requires caution, consider $\mathcal N$ for $G={\rm SU}(2)$. As we have seen in \S\ref{explicit}, modulo rotations in ${\rm SO}(1,2)$, solutions to the reduced Nahm--Schmid equations can be written as $T_i(t)=f_i(t)e_i$, where the $f_i$ are given by Jacobi elliptic functions. With the $e_i$ given by \eqref{basis}, such a solution is symmetric at $t$ if and only if $f_2(t)=0$. Thus the condition that such a solution belongs to $\mathcal N$ is equivalent to $f_2(0)=f_2(1)=0$. Proposition \ref{TD} impl
 ies then, however, 
 that $\mathcal T$ belongs to the degeneracy locus. Thus, for $G={\rm SU}(2)$, the hypersymplectic structure is degenerate at {\em every} point of $\mathcal N$.


\begin{thebibliography}{10}

\bibitem{AHH}
M.R. Adams,\ J. Harnad {\upshape and} J.C. Hurtubise:
\newblock {\em Isospectral Hamiltonian flows in finite and infinite dimensions. II. Integration of flows},
\newblock Comm. Math. Phys.  \textbf{134}:   555--585, 1990

\bibitem{AdlvMoe1}
M. Adler {\upshape and} P. van Moerbeke:
\newblock {\em Completely integrable systems, Euclidean Lie algebras, and Curves},
\newblock  Adv. Math. \textbf{38}:  267--317, 1980

\bibitem{AdlvMoe2}
{M. Adler {\upshape and} P. van Moerbeke}: 
\newblock {\em Linearization of Hamiltonian systems, Jacobi varieties and representation theory},
\newblock {Adv. Math.} \textbf{38}: 318--379, 1980

\bibitem{AtiyahHitchin:1988}
M.F. Atiyah {\upshape and} N.J. Hitchin:
\newblock {\em The Geometry and Dynamics of Magnetic Monopoles},
\newblock M. B. Porter Lectures. Princeton University Press, Princeton, NJ,
  1988
  
\bibitem{BE}
{T.N. Bailey {\upshape and} M.G. Eastwood}:
\newblock {\em Complex paraconformal manifolds???their differential geometry and twistor theory}, 
\newblock {Forum Math.} \textbf{3}(1): 61--103, 1991

\bibitem{Biel}
R.~Bielawski: 
\newblock{\em Reducible spectral curves and the hyperk\"ahler geometry of adjoint orbits},
\newblock J. London Math. Soc. \textbf{76}: 719--738, 2007

\bibitem{Bpre}
R.~Bielawski: 
\newblock{\em Nahm's, Basu-Harvey-Terashima's equations and Lie superalgebras},
\newblock preprint, arXiv:1503.03779

\bibitem{Biq}
O. Biquard: 
\newblock {\em Sur les \'equations de Nahm et les orbites coadjointes des groupes de Lie semi-
simples complexes}, Math. Ann.\textbf{ 304}: 253--276, 1996

\bibitem{ByrFriHEI}
P.F. Byrd {\upshape and} M.D. Friedman:
\newblock {\em Handbook of Elliptic Integrals for Engineers and Scientists}, 
\newblock 2nd Edition, Springer-Verlag, 1971

\bibitem{DancerSwann:1996} A. Dancer {\upshape and} A. Swann:
\newblock {\em Hyper-K\"ahler metrics associated to compact Lie groups},
\newblock {Math. Proc. Cambridge Philos. Soc.} \textbf{120}(1):  61--69, 1996

\bibitem{DanSwa} A. Dancer {\upshape and} A. Swann:
\newblock {\em Hypersymplectic manifolds},
In: {\em Recent Developments in Pseudo-Riemannian Geometry}. ESI Lect. Math. Phys., EMS, Zurich, 97--111, 2008

\bibitem{Donaldson:1984a}
S.K. Donaldson:
\newblock {\em {N}ahm's equations and the classification of monopoles},
\newblock Comm. Math. Phys., \textbf{96}(3): 387--407, 1984

\bibitem{Hart}
P. Hartman: 
\newblock {\em Ordinary Differential Equations},
\newblock  John Wiley and Sons, 1964

\bibitem{Hitchin:1983}
N.J. Hitchin:
\newblock {\em On the construction of monopoles},
\newblock { Comm. Math. Phys.}, \textbf{89}(2):145--190, 1983

\bibitem{Hitchin:1990} N.J. Hitchin:
\newblock{\em  Hypersymplectic quotients}, 
\newblock {Acta Acad. Sci. Tauriensis} \textbf{124}: 169--180, 1990

\bibitem{HKLR:1987}
N.J. Hitchin, A.~Karlhede, U.~Lindstr\"om {\upshape and} M.~Rocek:
\newblock{\em Hyperk\"ahler metrics and supersymmetry},
\newblock {Comm. Math. Phys.}  \textbf{108}(4): 535--589, 1987

\bibitem{Hull}
C.M. Hull:
\newblock {\em Actions for $(2,1)$ sigma models and strings},
\newblock \textsl{Nucl. Phys. B} \textbf{509}: 252--272, 1998

\bibitem{Kov}
A.G. Kovalev:
\newblock {\em Nahm's equations and complex adjoint orbits},
\newblock {Quart. J. Math. Oxford Ser.
(2)} \textbf{47}:  41--58, 1996

\bibitem{Kronheimer:1988}
P.B.~Kronheimer:
\newblock {\em A {H}yperk{\"a}hler structure on the cotangent bundle of a complex {L}ie group},
\newblock  MSRI preprint, 1988

\bibitem{KronJLMS}
P.B. Kronheimer:
\newblock {\em A hyper-k\"ahlerian structure on coadjoint orbits of a semisimple complex
group},
\newblock J. London Math. Soc. \textbf{42}: 193--208, 1990

\bibitem{Lieb}
P. Liebermann:
\newblock {\em Sur le probl\`eme d'\'equivalence de certaines structures infinit\'esimales},
\newblock {C. R. Acad. Sc. Paris.} \textbf{233}: 27--120, 1951

\bibitem{Mal}
A.N. Malyshev:
\newblock {\em Factorization of matrix polynomials}, 
\newblock {Siberian Math. J.}  \textbf{23}: 136--146, 1982

\bibitem{ManSut}
N. Manton {\upshape and} P. Sutcliffe:
\newblock {\em Topological Solitons},
\newblock Cambridge University Press, 2004

\bibitem{Matsoukas:2010}
T.~Matsoukas:
\newblock {\em Hypersymplectic Quotients},
\newblock DPhil Thesis, University of Oxford, 2010

\bibitem{Nahm}
W. Nahm:  
\newblock {\em All self-dual multimonopoles for arbitrary gauge groups},
\newblock \textrm{Preprint {\tt TH.3172-CERN} (unpublished)} 1981

\bibitem{RoeserDPhil}
M. R\"oser: 
\newblock{\em The ASD Equations in Split Signature and Hypersymplectic Geometry},
\newblock DPhil thesis, University of Oxford, 2012

\bibitem{Roeser:2014}
M. R\"oser:
\newblock {\em Harmonic Maps and Hypersymplectic Geometry},
\newblock {J. Geom. Phys.}  \textbf{78}: 111--126, 2014

\bibitem{Ros}
M. Rosenblatt:
\newblock {\em A multidimensional prediction problem}, 
\newblock {Arkiv Mat.} \textbf{3}: 407--424, 1958

\bibitem{Schm}
W. Schmid:
\newblock {\em Variation of Hodge structure: the singularities of the period mapping},
\newblock {Invent. Math.} \textbf{22}:  211--319, 1973

\bibitem{Ter}
{S. Terashima}:
\newblock {\em On M5-branes in $N = 6$ Membrane Action}, 
\newblock {JHEP} \textbf{0808:080}, 2008.
\end{thebibliography}
\end{document}